\documentclass[12pt,a4paper]{article}

\usepackage[utf8]{inputenc}
\usepackage[english]{babel}
\usepackage{amsmath}
\usepackage{color}
\usepackage{amsfonts,enumitem}
\usepackage{amssymb}
\usepackage{hyperref}
\usepackage{cleveref}
\usepackage{dsfont}
\usepackage{algorithm}
\usepackage{graphicx}
\usepackage{wrapfig}
\usepackage{fancyhdr}

\newcommand{\crit}{\mathrm{crit}}
\newcommand{\conv}{\mathrm{conv}}

\newcommand*{\SET}[1]  {\ensuremath{\mathbb{#1}}}
\newcommand{\argmax}{\operatornamewithlimits{argmax}}
\newcommand{\argmin}{\operatornamewithlimits{argmin}}
\newcommand{\R}{\SET{R}}

\newcommand{\N}{\SET{N}}

\newcommand{\D}{\mathbf{D}}
\newcommand{\B}{\SET{B}}

\newcommand{\X}{\mathcal{X}}
\newcommand{\Y}{\mathcal{Y}}

\newcommand{\mcU}{\mathcal{U}}
\newcommand{\V}{\mathcal{V}}

\newcommand{\Graph}{\operatorname{Graph}}
\newcommand{\Jac}{\operatorname{Jac}}
\newcommand{\proj}{\operatorname{proj}}

\newcommand{\SR}{\operatorname{SR}}
\newcommand{\W}{\operatorname{W}}
\newcommand{\T}{\operatorname{T}}
\newcommand{\SW}{\operatorname{SW}}
\newcommand{\FR}{\operatorname{FR}}

\newcommand{\ST}{\operatorname{ST}}
\newcommand{\di}{\text{d}}

\newtheorem{theorem}{Theorem}[section]
\newtheorem{lemma}{Lemma}[section]
\newtheorem{proposition}{Proposition}[section]
\newtheorem{corollary}{Corollary}[section]

\newtheorem{definition}{Definition}[section]
\newtheorem{assumption}{Assumption}
\newtheorem{remark}{Remark}[section]

\newenvironment{proof}[1][]{\noindent {\bf Proof #1:\;}}{\hfill $\Box$}

\title{Convergence of empirical subgradients for optimal transport-based objectives}

\author{
Tam Le\thanks{Université Paris Cité, LPSM}
}

\begin{document}
\maketitle
\begin{abstract}
Optimal transport is widely used to approximate distributions, enforce distributional constraints, and model uncertainty in machine learning. In applications, transport losses are often computed from samples through tractable representations, such as one-dimensional sorting formulas or sliced Wasserstein costs, making them practical components in training pipelines. We study parameterized objectives defined by sampled transport costs and prove graphical convergence of their subdifferentials to the subdifferential of the population objective. In particular, this ensures that standard subgradient methods consistently approach stationary points of the population-level problem. We illustrate the results in several settings, including risk-averse optimization, fairness-constrained learning, and sliced Wasserstein problems. Our analysis highlights that smooth parameterizations provide a favorable interface between statistical consistency and optimization. By contrast, transport objectives with nonsmooth costs and models may exhibit unstable derivatives in the large-sample limit.
\end{abstract}
\medskip


\medskip


\section{Introduction}

Optimal transport \cite{villani,peyre2019comput} has emerged as a central tool in computational mathematics and machine learning. It provides a principled framework for comparing probability measures, with important applications to sampling and generative modeling \cite{arjovsky2017wasserstein,generative_gulrajani2017improved}. It also offers a natural way to impose distributional constraints, for instance in fairness assessment \cite{fairness_risser2022tackling,fairness_dwork2012fairness,fairness_rychener2022metrizing}, to model uncertainty \cite{kuhn2019wasserstein,gao2023wassersteindistance}, and to encode risk-aversion in robust learning tasks \cite{risk_mehta2023stochastic,risk_bonalli2025characterization}. Efficient implementations of optimal transport are now well established, with differentiable programming frameworks available for practical use~\cite{cuturi2022optimal,flamary2021pot}.

In many learning problems, transport costs enter as losses depending on model parameters, rather than merely as distances between fixed distributions. A typical objective is
\[
    \theta \mapsto \T_c((h_\theta)_\#\mu,\nu),
\]
where \(\T_c\) denotes an optimal transport cost and \(h_\theta:\mathcal X\to\mathbb R^d\) is a parameterized model. If \(X\sim\mu\), then \((h_\theta)_\#\mu\) is the distribution of the model output \(h_\theta(X)\), which is compared to a target distribution \(\nu\). In practice, this population objective is typically approximated by sampling, or minibatching. Computing and differentiating these sampled objectives is standard in training pipelines. However, the behavior of the resulting gradient oracles is delicate, since optimal matchings introduce additional nonsmoothness.

\begin{wrapfigure}{r}{0.50\textwidth}
    \centering
    \vspace{-0.9em}
    \includegraphics[width=0.50\textwidth]{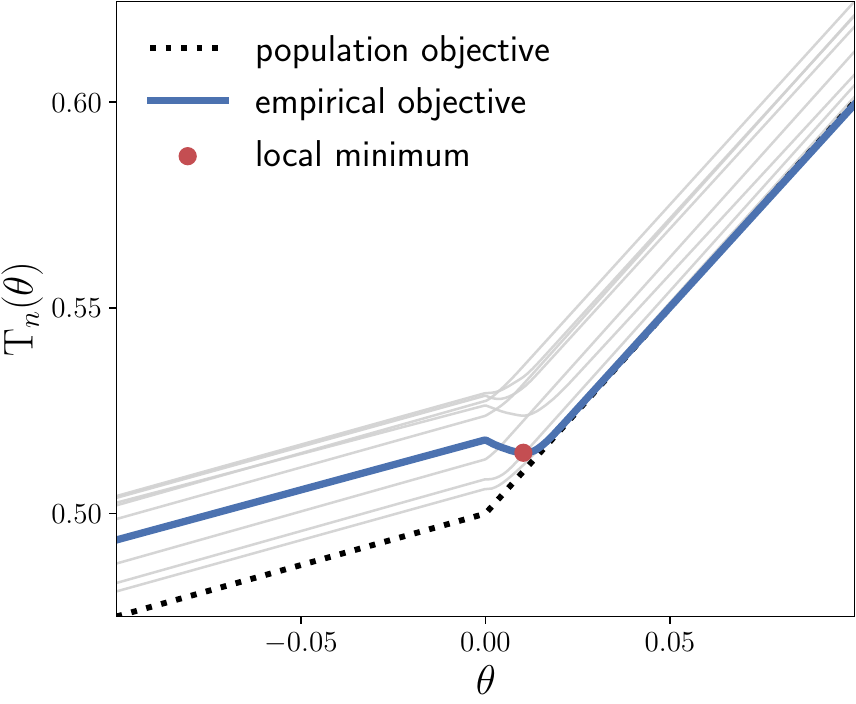}
    \vspace{-1.5em}
    \caption{\small Sample-induced local minimum}
    \label{fig:spurious-local-minimum}
    \vspace{-0.9em}
\end{wrapfigure}
One-dimensional transport costs illustrate this phenomenon well as they admit a quantile representation~\cite[Chapter 2]{Santambrogio2015}, which reduces to an explicit sorting-based formula for discrete measures. This links transport costs with ranks and quantiles and makes them particularly convenient in learning pipelines. Such tractability has been exploited for instance in image processing~\cite{image_delon2004midway}, risk measures and distributional robustness~\cite{spectral_risk_xiao2023unified,risk_bonalli2025characterization,fairness_risser2022tackling}. It is also the basis of sliced Wasserstein costs~\cite{sliced_rabin2011wasserstein}, which average tractable one-dimensional transport costs over random projections and have become a flexible tool in distributional learning tasks~\cite{sliced_deshpande2018generative,sliced_rodriguez2025learning,sliced_chapel2026differentiable,sliced_lobashev2026color,nguyen2024sliced}; see also~\cite{sliced_nadjahi2021sliced}.

This practical use of transport losses brings together two computational aspects. On the statistical side, population distributions are replaced by random finite samples, or minibatches~\cite{sampling_fatras2021minibatch,sampling_nadjahi2020statistical}, while on the optimization side, training algorithms rely on nonsmooth first-order information. Understanding this interaction is essential for providing guarantees for the use of optimal transport in applications.

\paragraph{Contributions and organization of the paper.}

We study the subdifferentials of transport costs between parameterized pushforward models. Our first main result, stated in \Cref{subsec:main_results_vanilla_transport}, shows that the subdifferentials of empirical transport costs converge to the subdifferential \eqref{eq:clarke_subdifferential} of the population limit. The convergence is understood in a \emph{graphical} sense, which is specific to set-valued maps. In particular, this yields convergence of the empirical critical sets, ensuring that subgradient-based methods consistently approach population-level stationary points \cite{benaim2012}. In \Cref{subsec:main_results_vanilla_sliced}, we extend the result to sliced costs. We present several examples in \Cref{section:2_applications}. Along the way, we present formulas that may be useful for practical implementations, including a compact expression for one-dimensional optimal transport between empirical measures with unequal sample sizes; see \Cref{lem:unbalanced_wasserstein}.

In \Cref{sec:graphical_convergence}, we recall the notions and results on graphical convergence used throughout the paper and state general limit results for set-valued maps defined through integrals against subsets of measures.  In \Cref{sec:4_convergence_subdiff_OT}, we establish our main results.  We highlight that subdifferentiating optimal transport costs can be reduced to a transport problem with parameter-dependent cost. We then establish the stability properties required to pass to the population limit and derive an envelope formula in \Cref{subsection:envelope_formula}. \Cref{subsection:proof_convergence_empirical_subdiff}-\ref{subsec:sliced_costs_proofs} establish the graphical convergence of the empirical subdifferentials for transport costs and sliced versions. Proofs of machine learning applications are found in the ending part,   \Cref{subsec:proof_applications}, where we apply our main theorems.

Our assumptions apply to continuously differentiable unit costs, e.g. \(p\)-Wasserstein costs with \(p>1\), and smooth models. Our study suggests that this smoothness is crucial for obtaining stable first-order approximations. In particular, in \Cref{sec:failure}, we showcase toy transport settings with nonsmooth unit cost and model where empirical objectives consistently develop a spurious local minimum; see \Cref{fig:spurious-local-minimum}.

\paragraph{Related works.}  
 Existing works on statistical limits for optimal transport have mostly focused on convergence of the transport value \cite{sampling_nadjahi2020statistical,fournier2015rate,sampling_fatras2021minibatch}. Less is known about the convergence of subgradients for objectives induced by parameterized models. Our results contribute to this direction and connect to several lines of research on (statistical) limits of subdifferential and differentiability of optimal transport.

\emph{Limits of subdifferentials.} The convergence of subgradients has been extensively studied in variational and nonsmooth analysis, often through graphical convergence~\cite{graphical_attouch2006convergence,graphical_levy1995partial}. This is the notion of convergence adopted in this work. It is well-suited to nonsmooth optimization, since it describes the limiting behavior of approximate set-valued first-order objects and provides a natural framework for analyzing the limit points of subgradient-based algorithms~\cite{benaim2012}. Several works at the interface of statistics and nonsmooth analysis establish graphical convergence for set-valued mappings and subdifferentials under empirical averaging~\cite{graphical_artstein1975strong,shapiro,graphical_salim2023strong}. Our work extends this line of analysis to parameterized optimal transport costs, where the nonsmoothness arises from both the transport problem and the dependence on model parameters.

A complementary part of the literature derives subdifferential convergence from convergence of the objective functions under uniform regularity assumptions
\cite{graphical_attouch2006convergence,graphical_levy1995partial,zolezzi1994convergence}. Our convergence results are related to these general theorems, and some of the final steps in our proof could likely also be recovered through semidifferentiability arguments~\cite{zolezzi1994convergence}. Our approach, however, makes explicit the transport mechanisms behind subgradient convergence, in particular the stability of optimal plans. Such stability is a classical theme in optimal transport~\cite{villani,stability_merigot2020quantitative,letrouit2024gluing}. In our setting, it appears as a parametric stability problem that we study qualitatively and relate to the question of graphical convergence. Our proofs also provide a complementary perspective by establishing optimal-transport stability with standard results in set-valued analysis~\cite{infinite,ghossoub2021continuity}.

\emph{Differentiability of transport costs.}
Nonsmooth differentiability plays a central role in our analysis, in particular through envelope-type \cite{danskin2012theory} formulas that identify the relevant first-order objects governing the large-sample limit. The differentiability of transport costs has been studied in several works, especially since the introduction of optimal-transport-based generative models~\cite{arjovsky2017wasserstein,generative_gulrajani2017improved,nonsmooth_houdard2023gradient}. In machine learning, this question has often been approached through differentiable approximations of optimal transport, such as entropic regularization and Sinkhorn iterations, which enable gradient-based optimization frameworks~\cite{sinkhorn_cuturi2014fast,sinkhorn_NEURIPS2019_d8c24ca8,wang2023sinkhorn} and come with corresponding guarantees~\cite{carlier2017convergence,sinkhorn_pauwels2023derivatives}. More recently, the nonsmoothness induced by sorting in sliced transport costs has attracted renewed attention~\cite{sliced_vauthier2025towards,sliced_chapel2026differentiable,sliced_tanguy2025properties}. We focus on exact transport costs, where, in dimension one, sorting yields efficient closed-form formulas and naturally generates nonsmooth parameterized objectives. This combination motivates a dedicated study of subgradient convergence for exact empirical transport losses. We emphasize that our assumptions are fairly general and explicit. In particular, our (sub)differentiation formula does not rely on entropic regularization, uniqueness of optimal transport solutions, or specific structural assumptions on the data distributions, beyond natural moment conditions.

\paragraph{Notations and essential definitions.}

For a function $\psi$, the $a$-sublevel set is denoted by $\{\psi \le a\}$; analogous notations are used for level and superlevel sets.  
The closed unit ball centered at the origin is written $\overline{\mathbb{B}}$, and $\mathbb{S}^{p-1} \subset \mathbb{R}^p$ denotes the $(p-1)$-dimensional unit sphere.  
The Euclidean norm is denoted by \(\|\cdot\|\); for matrices, the same notation refers to the induced operator norm. For a compact set \(A\subset\mathbb R^p\), we write $\|A\|:=\max_{y\in A}\|y\|.$

For a subset $C \subset \mathbb{R}^p$, we denote its closure, interior, and boundary by $\overline{C}$, $\operatorname{int}(C)$, and $\partial C$, respectively.  The convex hull of a set is denoted by $\operatorname{conv} C$. 
The distance from a point $x$ to $C$ is $\operatorname{dist}(x,C)$, and for compact sets $A,B \subset \mathbb{R}^p$, the \emph{excess distance} is
\begin{equation}
\label{eq:excess_distance}
    \mathbf{D}(A,B) := \sup_{a \in A} \operatorname{dist}(a,B).
\end{equation}

\emph{Subdifferential.} In this work, we use a generalized gradient for nonsmooth functions. Let $f : \R^p \to \R$ be locally Lipschitz.  
Then $f$ is differentiable on a set of full Lebesgue measure $\Delta_f \subset \R^p$, and its \emph{Clarke subdifferential} at $\theta$ is defined by the convex compact set
\begin{equation}
    \label{eq:clarke_subdifferential}
    \partial f(\theta)
:=
\operatorname{conv}\!\left\{
v \in \R^p :
\exists\, \{x_k\}_{k\in\mathbb{N}} \subset \Delta_f,\;
v = \lim_{k\to\infty} \nabla f(x_k)
\right\}.
\end{equation}
Elements of the subdifferential are called \emph{subgradients}.

Let $\Theta \subset \R^p$ be a convex compact set. Critical points of $f$ on $\Theta$ are defined as 
\begin{equation}
\label{eq:critical_points}
    \crit_\Theta f := \{x \in \R^p \ : \ 0 \in \partial f(x) + N_\Theta(x) \}
\end{equation}
where $N_\Theta(x)$ is the normal cone to $\Theta$.

\emph{Measure spaces.} For a Polish space \(S\), that is, a separable complete metric space, we denote by \(\mathcal P(S)\) the set of Borel probability measures on \(S\). Throughout the paper, all underlying measurable spaces are assumed to be Polish, endowed with their Borel \(\sigma\)-algebras. Given two measures $\mu$ and $\nu$, $\Pi(\mu,\nu)$ is the set of couplings with marginals $\mu$ and $\nu$.  $\delta_a$ denotes the Dirac mass at a point $a$. For a real-valued random variable $X$, we denote by $F_X$ its cumulative distribution function and by $F_X^{-1}$ its quantile function. When $h$ is a measurable map, we denote by $h_\# \mu$ the pushforward of $\mu$ by $h$, that is distribution of $h(X)$ under $X \sim \mu$. $\operatorname{Unif}(a,b)$ is the uniform distribution over a segment $[a,b]$.


\section{Main results}
\label{sec:main_results}

\subsection{Convergence of subgradients for differentiable unit costs} \label{subsec:main_results_vanilla_transport}

We fix distributions $\mu \in \mathcal{P}(\X)$ and $\nu \in \mathcal{P}(\Y)$. Let $f : \R^p \times \X \to \mathcal{U}$ and $g : \R^p \times \Y \to \mathcal{V}$ be parameterized families,  where $\mathcal{U}, \mathcal{V} \subset \R^d$ are output spaces.  Given a cost function $c : \mathcal{U} \times \mathcal{V} \to \R$,  we study the transport cost between the pushforward distributions,
\begin{equation}
\label{eq:transport_cost_population}
\T_c(\theta)
=
\inf_{\pi \in \Pi((f_\theta)_\# \mu, (g_\theta)_\# \nu)}
\int_{\mathcal{U} \times \mathcal{V}} c(u,v)\, \di \pi(u,v),
\end{equation}
with the notations $f_\theta := f(\theta,\cdot)$ and $g_\theta := g(\theta,\cdot)$. In practice, \(\mu\) and \(\nu\) are observed through samples. Hence let \((x_1,\ldots,x_n)\) be i.i.d. samples from \(\mu\), and \((y_1,\ldots,y_m)\) i.i.d. samples from \(\nu\). Denoting the associated empirical measures $\mu_n$ and $\nu_m$ respectively, we consider the empirical transport objective and the Clarke subgradients, whenever they are defined,
\begin{equation*}
\label{eq:transport_cost_empirical}
\T_c^{n,m}(\theta)
:=
\inf_{\pi \in \Pi((f_\theta)_\#\mu_n,\, (g_\theta)_\#\nu_m)}
\int_{\mathcal{U} \times \mathcal{V}} c(u,v)\, \di\pi(u,v), \qquad
g_{n,m} \in \partial \T_c^{n,m}(\theta).
\end{equation*}

This setting captures several practical situations. For instance, when $c$ is a distance, $g_\theta$ is fixed as the identity, and $f_\theta$ is a parameterized model, this may encode a generative modelling task, as in Wasserstein generative models \cite{arjovsky2017wasserstein} and encoding problems \cite{sebbouh2022randomized,sliced_rodriguez2025learning}. When $\mu$ and $\nu$ represent the distribution of groups of individuals, and $f$ and $g$ are predictions, the transport cost may be incorporated as an additional penalty to a learning objective to enforce fairness \cite{fairness_risser2022tackling,fairness_rychener2022metrizing}. Finally, an appropriate choice of target $\nu$ with a scalar product cost can model risk-aversion \cite{risk_mehta2023stochastic}. These concrete examples will be exposed in detail in \Cref{section:2_applications}.

We now state our main convergence result for subgradients of general transport costs. To this end, define the parameterized sample-level cost
\begin{equation}
    \label{eq:c_f_g}
    C(\theta,x,y):=c(f(\theta,x),g(\theta,y))
\end{equation}
which expresses the transport cost directly in terms of the underlying samples. We work under the following assumptions.
\begin{assumption}
\label{ass:differentiability_general_transport_cost}
\begin{enumerate}
    \item[]
    \item There exist open sets $E \subset \X$ and $F \subset \Y$ with $\mu(E) =  \nu(F)= 1$, such that $C$ and $\nabla_\theta C$ are jointly continuous on $\R^p \times E \times F$.
    \item There exist a locally bounded function $\kappa : \R^p \to \R_+$, measurable functions $\psi_\X : \X \to \R_+$, $\psi_\Y : \Y \to \R_+$, and $\eta > 0$ such that $\int_{\X} \psi_{\X}^{1+\eta} \di \mu, \int_{\Y} \psi_\Y^{1+ \eta} \di \nu<\infty$ and for all $(\theta,x,y) \in \R^p \times \X \times \Y$, $|C(\theta, x,y)| +\|\nabla_\theta C(\theta,x,y)\| \leq \kappa(\theta)\bigl(\psi_\X(x) + \psi_\Y(y)\bigr)$.
\end{enumerate}
\end{assumption}
These conditions cover, in particular, \(q\)-Wasserstein costs with \(q>1\), and continuously differentiable models. The continuity requirement on \(\nabla_\theta C\) is also compatible with discrete distributions: it only needs to hold pointwise in the discrete variables. More generally, the assumption allows for mixed continuous-discrete settings. Note that the assumptions do not cover Wasserstein 1 distances and nonsmooth models. These excluded cases are examined through pathological examples in \Cref{sec:failure}, suggesting that the assumptions are close to tight.

The starting point of our analysis is a differentiation formula for subgradients of transport costs.

\begin{lemma}[Parameterized cost differentiation] Under \Cref{ass:differentiability_general_transport_cost}, where $C$ is given by \eqref{eq:c_f_g}, for all $\theta \in \R^p$, 
\begin{equation*}
    \partial \T_c(\theta) = \left\{ \int_{\X \times \Y} \nabla_\theta C(\theta,x,y) \di \gamma(x,y) \ : \  \gamma \in \argmin_{\gamma' \in \Pi(\mu,\nu)} \int_{\X \times \Y} C(\theta,x,y) \di \gamma'(x,y) \right\}
\end{equation*}
\end{lemma}
This is established in \Cref{prop:envelope_ot}. Subgradients with respect to the parameter are thus obtained by evaluating the gradient of the composite cost function $C$ in \eqref{eq:c_f_g} against sample-level optimal plans. Compared to existing studies \cite{arjovsky2017wasserstein,nonsmooth_houdard2023gradient}, this formula holds under general conditions and does not require regularization or technical assumptions to ensure differentiability.

Note that the formula also applies when $\mu$ and $\nu$ are replaced by empirical approximations. Convergence of the subgradient is then obtained by passing to the graphical limit in this formula. This is done in \Cref{sec:4_convergence_subdiff_OT} to obtain our main result.

\begin{theorem}[Convergence of empirical subgradients] \label{th:limit_subdiff} Let $\Theta \subset \R^p$ be a compact subset. Let $C$ be defined as \eqref{eq:c_f_g}. Under \Cref{ass:differentiability_general_transport_cost}, almost surely in the samples $(x_i)_{i \in \N}, (y_j)_{j \in \N}$, for any $\varepsilon >0$, there exists  $N_\varepsilon \geq 0$ such that for all $n,m \geq N_\varepsilon$,  $\partial \T_c^{n,m}(\theta) \subset \partial \T_c(\theta + \varepsilon \Bar{\B}) + \varepsilon \Bar{\B}$  for all $\theta \in \Theta$.
\end{theorem}
See \Cref{subsection:proof_convergence_empirical_subdiff} for the proof. Equivalently, the result can be expressed as convergence, in excess distance \eqref{eq:excess_distance}, of the corresponding graphs \eqref{eq:graph}; see \Cref{sec:graphical_convergence} for the relevant notions concerning graphs of set-valued maps.

\subsection{Extension to sliced costs}
\label{subsec:main_results_vanilla_sliced}
We extend the preceding approximation result to \emph{sliced transport costs}. To this end, set $\mathcal{U} = \mathcal{V} = \R^m$ and let $\rho$ be a probability distribution supported on a compact set $\mathcal{K} \subset \R^m$, typically chosen as the unit sphere. Given a cost function $c : \R \times \R \to \R$ and a projection direction $\phi \sim \rho$, we define the projected transport cost as
\begin{equation}
\label{eq:main_little_transport_c}
t_c(\theta,\phi)
=
\inf_{\pi \in \Pi((\phi^\top f_\theta)_\# \mu, (\phi^\top g_\theta)_\# \nu)}
\int_{\R \times \R} c(u,v)\di \pi(u,v),
\end{equation}
where $\phi^\top \cdot$ is the scalar product against $\phi$. The sliced transport cost between the parameterized models $f$ and $g$ is obtained by integrating over $\phi \sim \rho$.
\begin{equation*}
\ST_c(\theta)
:=
\int_{\mathcal{K}} t_c(\theta,\phi)\di \rho(\phi).
\end{equation*}

For i.i.d. samples $\phi_1,\ldots,\phi_k \sim \rho$, $x_1,\ldots,x_n \sim \mu$, and $y_1,\ldots,y_m \sim \nu$, with associated empirical measures $\rho_k$, $\mu_n$, and $\nu_m$, we consider the empirical counterpart
\begin{equation*}
\ST_c^{n,m,k}(\theta)
:=
\int_{\mathcal{K}}
t_c\left((\phi^\top f_\theta)_\# \mu_n, (\phi^\top g_\theta)_\# \nu_m\right)
\di \rho_k(\phi).
\end{equation*}
 Under suitable conditions on the associated sample-level cost \eqref{eq:sliced_parameterized_cost}, we obtain convergence of the empirical subgradients.

\begin{equation}
\label{eq:sliced_parameterized_cost}
    C: (\theta,x,y, \phi) \mapsto c(\phi^\top f(\theta,x),  \phi^\top g(\theta,y) )  
\end{equation}

\begin{assumption}[Differentiable parameterized sliced cost]
\label{assumption:main_differentiability_parameterized_sliced_cost}
\begin{enumerate}
\item[]
    \item (Regularity off null sets) 
    There exist open subsets $E \subset \X$ and $F \subset \Y$ with $\mu(E) = \nu(F) = 1$ such that $C$ and $ \nabla_\theta C$ are  continuous on $\R^p \times E \times F \times \mathcal K$.
    \item (Uniform integrability)  
    There exist a locally bounded function $\kappa : \R^p \to \R_+$, measurable functions $\psi_\X : \X \to \R_+$, $\psi_\Y : \Y \to \R_+$, and $\eta > 0$ such that $\int_\X \psi_\X^{1+\eta} \, \di \mu$, $\int_\Y \psi_\Y^{1+\eta} \, \di \nu < \infty$
    and for all $(\theta,x,y,\phi) \in \R^p \times \X \times \Y \times \mathcal K$, $$|C(\theta,x,y,\phi)| + \|\nabla_\theta C(\theta,x,y,\phi)\|
        \le \kappa(\theta)\bigl(\psi_\X(x) + \psi_\Y(y)\bigr).$$
\end{enumerate}
\end{assumption}

\begin{theorem}[Convergence of empirical subgradients for sliced cost] \label{th:limit_subdiff_sliced} Let $\Theta \subset \R^p$ be a compact subset. Let $C$ be defined as \eqref{eq:sliced_parameterized_cost}. Under \Cref{assumption:main_differentiability_parameterized_sliced_cost}, almost surely in the samples $(x_i)_{i \in \N}, (y_j)_{j \in \N}, (\phi_l)_{l \in \N}$, for any $\varepsilon >0$, there exists $N_\varepsilon \geq 0$ such that for any $n,m \geq N_\varepsilon$, there exists $K_{n,m} \geq 0$ such that for any $k \geq K_{n,m}$, $\partial \ST_c^{n,m,k}(\theta) \subset \partial \ST_c(\theta + \varepsilon \Bar{\B}) + \varepsilon \Bar{\B}$ for all $\theta \in \Theta.$
\end{theorem}

The proof is found in \Cref{subsec:sliced_costs_proofs}.

\subsection{Subgradient computation and convergence in transport-based learning}
\label{section:2_applications}

In this part, we present applications of our general convergence theorems, drawing on existing works: a risk-averse optimization problem, a fair learning framework, and a sliced Wasserstein learning problem. The proofs are deferred to \Cref{subsec:proof_applications}.

In each example, an empirical objective $\hat{F}$ is used as an empirical approximation of a population objective $F$. Our approximation results imply that stationary solutions of the finite-sample problem are reliable approximations of stationary solutions of the population problem. This critical-point stability complements existing convergence guarantees for projected subgradient methods and related nonsmooth optimization algorithms
\cite{norkin1980,Davis2020,bolte2019conservative}. For example, consider the projected iteration
\begin{equation*}
    w_{k+1}
    =
    \proj_\Theta\bigl(w_k-\eta_k \hat{G}(w_k)\bigr),
    \qquad
    \hat{G}(w_k)\in\partial \hat{F}(w_k).
\end{equation*}
where $\eta_k >0$ are stepsizes. For a fixed dataset, existing optimization theory describes the convergence of this subgradient algorithm toward critical points of the nonconvex and nonsmooth empirical objective $\hat{F}$, in the sense of \eqref{eq:critical_points}. Our results address the complementary statistical question: when the sample size is large, these empirical critical points are close to the critical set of the population objective $F$.

\subsubsection{Spectral risk minimization}
\label{subsection:spectral_risk}
We consider a \emph{spectral risk}-based learning framework 
\cite{risk_mehta2023stochastic,spectral_risk_laguel2022superquantile,spectral_risk_xiao2023unified}. 
Let $\ell : \Theta \times \X \to \mathbb{R}$ be a loss function, and let $\mu$ be a data distribution on $\X$. 
Given a weighting function $w : [0,1] \to \mathbb{R}_+$, we define the population spectral risk associated with the pushforward distribution $(\ell_\theta)_{\#}\mu$ by
\begin{equation}
    \label{eq:SR_population}
    \SR(\theta)
    =
    \int_{0}^1
    w(s)\,
    F^{-1}_{(\ell_\theta)_{\#}\mu}(s)
    \,\di s,
\end{equation}
where $\ell_\theta := \ell(\theta,\cdot)$ and $F^{-1}$ is the quantile function. Given i.i.d. samples $x_1,\ldots,x_n \sim \mu$, the corresponding empirical spectral risk is
\begin{equation}
    \label{eq:SR_empirical}
    \SR_n(\theta)
    =
    \frac{1}{n}
    \sum_{i=1}^n
    w(i/n)\,
    \ell_{(i)}(\theta),
\end{equation}
where, for each $\theta \in \Theta$, $\ell_{(1)}(\theta)
\leq \cdots \leq
\ell_{(n)}(\theta)$
denote the ordered values of $\ell(\theta,x_i)$, $i = 1, \ldots,n$. Note that the indices $(i)$ depend on $\theta$. In the presence of ties, we fix an arbitrary ordering convention. When $w$ is non-decreasing, the empirical objective \eqref{eq:SR_empirical} assigns greater weight to larger loss values, thereby promoting risk-averse minimization. Particular choices are $w(s) = \frac{\mathds{1}_{[1 - \alpha,1]}(s)}{1 - \alpha}$ with $\alpha \in (0,1)$ for $\alpha$-superquantile (or CVaR), and $w(s) = r s^{r-1}$ with $1\leq r<2$ for $r$-extremile, see e.g \cite{risk_mehta2023stochastic}.

Formally differentiating through the sorting operation yields the gradient oracle
\begin{equation}
    \label{eq:gradient_oracle_spectral_risk}
    G^{\SR}_n(\theta)
    :=
    \frac{1}{n}
    \sum_{i=1}^n
    w(i/n)\,
    \nabla \ell_{(i)}(\theta).
\end{equation}

We can show \eqref{eq:gradient_oracle_spectral_risk} is a subgradient of the empirical spectral risk objective. This fact is known and can be established by more elementary arguments; see, for example, 
\cite[Property 2]{pillutla2024federated}; our analysis places it within a broader framework and relaxes several assumptions commonly imposed in earlier treatments, including gradient Lipschitz continuity and boundedness of the loss. As a consequence of \Cref{th:limit_subdiff}, we obtain convergence of the gradient oracle in 
\eqref{eq:gradient_oracle_spectral_risk}. 
This holds under the following conditions.

\begin{assumption}
\label{ass:differentiability_spectral_risk}
\begin{enumerate}
    \item[]
    \item $w : [0,1] \to \R_+$ is non-decreasing, bounded and continuous on an open set of full Lebesgue measure, $\ell$ is measurable, for almost all $x \in \X$, $\ell(\cdot, x)$ is differentiable and $\nabla_\theta \ell$ is jointly continuous.
    \item     There exist a non-negative and locally  bounded function $\kappa : \R^p \to \R_+$, a measurable function $\psi : \X \to \R_+$ and $\eta > 0$ such that    $\int_\X \psi^{1+\eta}  \di \mu < \infty$ and for all $(\theta,x) \in \R^p \times \X$, $|\ell(\theta,x)| + \|\nabla_\theta \ell(\theta,x)\| \leq \kappa(\theta)  \psi(x).$
\end{enumerate}
\end{assumption}
\begin{corollary} \label{cor:limit_spectral_risk} Let $\Theta \subset \R^p$ be a compact subset. Under \Cref{ass:differentiability_spectral_risk}, for all $n \geq 1$, $G^{\SR}_n \in \partial \SR_{n}$ and almost surely, for any $\varepsilon > 0$, there exists $N_\varepsilon \geq 0$ such that for all $n \geq N_\varepsilon$,
\begin{enumerate}
    \item $\partial \SR_n(\theta) \subset \partial \SR(\theta + \varepsilon \Bar{\B}) + \varepsilon \Bar{\B}$ for all $\theta \in \Theta$.
    \item $\crit_\Theta \SR_n \subset (\crit_\Theta \SR) + \varepsilon \Bar{\B}$.
\end{enumerate}
\end{corollary}

\subsubsection{Fair learning}

We next consider an application to fair learning. To promote fairness with respect to a sensitive attribute, such as gender or demographic group, several works impose \emph{distributional parity} constraints on a score function \cite{fairness_feldman2015certifying,fairness_risser2022tackling,fairness_rychener2022metrizing}. Depending on the choice of score, this may encode, for instance, parity of losses or parity of predictions. Let $\mu^0$ and $\mu^1$ denote the data distributions associated with two sensitive groups, and let $s : \Theta \times \X \to \R$ be a score function. Population-level discrepancy between the two score distributions may be measured through the one-dimensional quadratic transport cost
\begin{equation}
    \label{eq:FR_population}
    \FR(\theta)
    :=
    \frac{1}{2}
    \int_0^1
    \left(
    F^{-1}_{(s_\theta)_\# \mu^0}(u)
    -
    F^{-1}_{(s_\theta)_\# \mu^1}(u)
    \right)^2
    \,\di u,
\end{equation}
where $s_\theta := s(\theta,\cdot)$. This quantifies discrepancies between the score distributions of the two groups. Given i.i.d  samples $x_1^j,\ldots,x_m^j \sim \mu^j,$ $j\in\{0,1\},$ assumed mutually independent, let $\mu_m^0$ and $\mu_m^1$ denote the corresponding empirical measures. For each $\theta \in \Theta$, we write
$s^j_{(1)}(\theta)
\leq \cdots \leq
s^j_{(m)}(\theta)$ for the ordered values of $s(\theta,x_1^j),\ldots,s(\theta,x_m^j).$ The empirical fairness regularizer is then
\begin{equation}
\label{eq:FR_empirical}
    \FR_m(\theta)
    :=
    \frac{1}{2m}
    \sum_{i=1}^m
    \left(
    s^0_{(i)}(\theta)
    -
    s^1_{(i)}(\theta)
    \right)^2.
\end{equation}
Equivalently, $\FR_m(\theta)$ is the empirical transport cost between $(s_\theta)_\#\mu_m^0$ and $(s_\theta)_\#\mu_m^1$ for the cost $c(u,v)=\frac{1}{2}(u-v)^2.$ This regularizer can then be incorporated into empirical risk minimization to promote fairness during training. Given a loss function $\ell : \Theta \times \X \to \R$, define $\mathcal{L}_n(\theta)
:=
\frac{1}{n}
\sum_{i=1}^n
\ell(\theta,x_i),$
and consider the regularized objective
\begin{equation*}
\label{eq:fair_regularized_erm}
    \min_{\theta\in\Theta}
    \mathcal{L}_n(\theta)
    +
    \lambda \FR_m(\theta),
\end{equation*}
where $\lambda>0$ controls the strength of the fairness penalty. Formally differentiating \eqref{eq:FR_empirical} through the sorting operation yields the gradient oracle
\begin{equation}
    \label{eq:gradient_oracle_fairness}
    G_m^{\FR}(\theta)
    :=
    \frac{1}{m}
    \sum_{i=1}^m
    \left(
    s^0_{(i)}(\theta)
    -
    s^1_{(i)}(\theta)
    \right)
    \left(
    \nabla_\theta s^0_{(i)}(\theta)
    -
    \nabla_\theta s^1_{(i)}(\theta)
    \right).
\end{equation}
As in the spectral-risk setting, this oracle is a subgradient of the empirical objective. Moreover, our general approximation result yields convergence of empirical subgradients toward the subdifferential of the population fairness regularizer \eqref{eq:FR_population}. We impose the following assumptions.

\begin{assumption}
\label{ass:differentiability_fairness}
\begin{enumerate}
    \item[]
    \item The score function $s$ is measurable, $s(\cdot,x)$ is differentiable for $\mu^0$- and $\mu^1$-almost every $x\in\X$, and $\nabla_\theta s$ is jointly continuous.
    
    \item There exist a non-negative locally bounded function
    $\kappa : \R^p \to \R_+$,  a measurable function $\psi : \X \to \R_+$ and $\eta>0$, such that $\int_\X \psi^{2(1+\eta)}\,\di\mu^j <\infty,$ $j\in\{0,1\},$ and, for all $(\theta,x)\in\R^p\times\X$, $|s(\theta,x)|
    +
    \|\nabla_\theta s(\theta,x)\|
    \leq
    \kappa(\theta) \psi(x).$
    
    \item The loss function $\ell$ is measurable, $\ell(\cdot,x)$ is differentiable for $\mu$-almost every $x\in\X$, and $\nabla_\theta\ell(\cdot,x)$ is continuous. Moreover, there exist a locally bounded function
    $r:\R^p\to\R_+$ and a $\mu$-integrable function
    $K:\X\to\R_+$ such that $\|\nabla_\theta \ell(\theta,x)\|
    \leq
    r(\theta)K(x)$ for all $(\theta,x)\in\R^p\times\X$.
\end{enumerate}
\end{assumption}

The following result is a direct consequence of \Cref{th:limit_subdiff}.

\begin{corollary}
\label{cor:limit_fairness}
Let $\Theta\subset\R^p$ be compact. Under
\Cref{ass:differentiability_fairness}, for all $m\geq 1$, $G_m^{\FR}\in \partial \FR_m.$
Furthermore, almost surely, for every $\varepsilon>0$, there exists
$N_\varepsilon\geq 0$ such that, for all $m\geq N_\varepsilon$,
\begin{enumerate}
    \item $\partial \FR_m(\theta)
    \subset
    \partial \FR(\theta+\varepsilon\Bar{\B})
    +
    \varepsilon\Bar{\B},$ $\text{for all }\theta\in\Theta.$

    \item Writing $\mathcal{L}(\theta)
    :=
    \int_\X \ell(\theta,x)\,\di\mu(x),$
    one has $\partial\bigl(\mathcal{L}_n+\lambda\FR_m\bigr)(\theta)
    \subset
    \partial\bigl(\mathcal{L}+\lambda\FR\bigr)
    (\theta+\varepsilon\Bar{\B})
    +
    \varepsilon\Bar{\B},$ $\text{for all }\theta\in\Theta.$

    \item $\crit_\Theta\bigl(\mathcal{L}_n+\lambda\FR_m\bigr)
    \subset
    \crit_\Theta\bigl(\mathcal{L}+\lambda\FR\bigr)
    +
    \varepsilon\Bar{\B}.$
\end{enumerate}
\end{corollary}

\paragraph{Unbalanced group samples.} In fair learning problems, the sensitive groups are often unbalanced. We therefore state a closed-form expression for the empirical transport penalty when the group sizes $n_0$ and $n_1$ are not necessarily equal. A related formula appears in \cite[Proposition 3.1]{sliced_rodriguez2025learning}; the version proposed below is slightly more compact. See the extended version in \Cref{lem:unbalanced_wasserstein_extended} for a proof.

\begin{lemma}[Wasserstein cost for unequal sample sizes]
\label{lem:unbalanced_wasserstein}
Let $\mu_U = \frac1n \sum_{i=1}^n \delta_{U_i}$ and $\mu_V = \frac1m \sum_{j=1}^m \delta_{V_j}$ be empirical measures on $\R$, with order statistics $U_{(1)} \le \cdots \le U_{(n)}$ and $V_{(1)} \le \cdots \le V_{(m)}$.  
Let $0 = h_0 \leq h_1 \leq \cdots \leq h_{n+m} = 1$ be the sorted values of $\{k/n\}_{k=0}^n \cup \{l/m\}_{l=0}^m$, and set $\Delta_i^{n,m} = h_i - h_{i-1}$. Then for any $q \ge 1$,
\[
 \int_{0}^1|F^{-1}_U(s) - F^{-1}_{V}(s)|^q \di s =
\sum_{i=1}^{n+m}
\bigl|U_{(\lceil n h_i \rceil)} - V_{(\lceil m h_i \rceil)}\bigr|^q \,\Delta_i^{n,m}.
\]
\end{lemma}

\begin{remark}
The sequence $(h_i)_{i=0}^{n+m}$ may contain repeated values whenever
$\{k/n\}_{k=0}^n$ and $\{\ell/m\}_{\ell=0}^m$ intersect. In that case,
the corresponding increments $\Delta_i^{n,m}$ vanish and removing such vacuous terms in practice yields an equivalent formula.
\end{remark}
For empirical group measures $\mu_{n_0}^0$ and $\mu_{n_1}^1$, the fairness regularizer becomes

\[
\operatorname{FR}_{n_0,n_1}(\theta)
= \frac{1}{2}
\sum_{i=1}^{n_0+n_1}
\big(
s_{(\lceil n_0 h_i \rceil)}^0(\theta)
-
s_{(\lceil n_1 h_i \rceil)}^1(\theta)
\big)^2
\Delta_i^{n_0,n_1},
\]
where $(h_i, \Delta_i^{n_0,n_1})_{i=0, \ldots,n_0+n_1}$  are given by \Cref{lem:unbalanced_wasserstein}. Consider the oracle
\begin{equation*}
G_{n_0,n_1}^{\operatorname{FR}}(\theta)
=
\sum_{k=1}^{n_0+n_1}
\big(
s^0_{(\lceil n_0 h_k \rceil)}(\theta)
-
s^{1}_{(\lceil n_1 h_k \rceil)}(\theta)
\big)
\big(
\nabla s^0_{(\lceil n_0 h_k \rceil)}(\theta)
-
\nabla s^1_{(\lceil n_1 h_k \rceil)}(\theta)
\big)
\Delta_k^{n_0,n_1}.
\end{equation*}
The convergence result obtained in the balanced setting extends directly to unequal sample sizes, that is, as $n_0,n_1\to\infty$, the empirical fairness subgradients graphically approximate the subdifferential of the population regularizer $\FR$.

\begin{corollary} \label{cor:limit_subdiff_fair} Let $\Theta \subset \R^p$ be a compact subset. Under \Cref{ass:differentiability_fairness}, $G_{n_0, n_1}^{\FR} \in \partial \FR_{n_0, n_1}$ and almost surely, for any $\varepsilon > 0$, there exists $N_\varepsilon \geq 0$ such that for all $n_0, n_1 \geq N_\varepsilon$, 

\begin{enumerate}
    \item $\partial \FR_{n_0,n_1}(\theta) \subset \partial \FR (\theta + \varepsilon \Bar{\B}) + \varepsilon \Bar{\B} \quad \text{for all } \theta \in \Theta.$
    \item $\partial\bigl(\mathcal{L}_n+\lambda\FR_{n_0,n_1}\bigr)(\theta)
    \subset
    \partial\bigl(\mathcal{L}+\lambda\FR\bigr)
    (\theta+\varepsilon\Bar{\B})
    +
    \varepsilon\Bar{\B},$ $\text{for all }\theta\in\Theta.$ Furthermore, $\crit_\Theta (\mathcal{L}_n + \lambda\FR_{n_0,n_1}) \subset \crit_\Theta (\mathcal{L} + \lambda\FR) + \varepsilon \Bar{\B}$. 
\end{enumerate}
\end{corollary}
The proof of the result is found in \Cref{subsec:proof_applications}. The inclusion $G_{n_0, n_1}^{\FR} \in \partial \FR_{n_0, n_1}$ requires a careful investigation and is established in \Cref{lem:correctness_fair}, while the two subsequent items are mainly consequences of the general \Cref{th:limit_subdiff}.
\subsubsection{Sliced Wasserstein learning}
\label{subsection:sliced_wasserstein}

We now illustrate the sliced transport framework of
\Cref{subsec:main_results_vanilla_sliced} through a distribution learning task. Let $f:\Theta\times\X\to\R^d$
be a parameterized generator, let $\mu$ be a source distribution on $\X$, and let $\nu$ be a target distribution on $\R^d$. We take $g(\theta,y)=y,$ $c(u,v)=\frac12 |u-v|^2,$ and let $\rho$ be a slicing distribution on $\mathbb{S}^{d-1}$. The resulting population objective is
\begin{equation*}
\label{eq:sliced_wasserstein_population}
    \SW(\theta)
    :=
    \int_{\mathbb{S}^{d-1}}
    \W_2^2\!\left(
        (\phi^\top f_\theta)_\#\mu,\;
        (\phi^\top \cdot)_\#\nu
    \right)
    \,\di\rho(\phi),
\end{equation*}
where $\W_2^2(\alpha,\beta)
:=
\inf_{\pi\in\Pi(\alpha,\beta)}
\int_{\R\times\R}
\frac12 |u-v|^2
\,\di\pi(u,v).$ Given i.i.d.\ directions $\phi_1,\ldots,\phi_k\sim\rho,$
source samples $x_1,\ldots,x_n\sim\mu,$
and target samples $y_1,\ldots,y_n\sim\nu,$ the empirical objective admits the computable formula
\begin{equation}
\label{eq:sliced_wasserstein_empirical}
    \SW_{n,k}(\theta)
    =
    \frac{1}{2nk}
    \sum_{j=1}^k
    \sum_{i=1}^n
    \left\langle
        \phi_j,\,
        f_{(i|\phi_j)}(\theta)-y_{(i|\phi_j)}
    \right\rangle^2.
\end{equation}
Here, for each direction $\phi$, the indices $(i|\phi)$ order the projected values $\phi^\top f(\theta,x_i)$  and $\phi^\top y_i$ for $i=1,\ldots,n$
respectively. As in previous applications, in the presence of ties, we fix an arbitrary ordering convention that does not affect our analysis. Differentiating \eqref{eq:sliced_wasserstein_empirical} through the sorting operation yields the subgradient oracle
\begin{equation}
\label{eq:gradient_oracle_sliced}
    G_{n,k}^{\SW}(\theta)
    :=
    \frac{1}{nk}
    \sum_{j=1}^k
    \sum_{i=1}^n
    \left\langle
        \phi_j,\,
        f_{(i|\phi_j)}(\theta)-y_{(i|\phi_j)}
    \right\rangle
    \Jac_\theta f_{(i|\phi_j)}(\theta)^\top \phi_j.
\end{equation}
The general sliced-cost result applies under the following conditions.
\begin{assumption}
\label{ass:differentiability_sliced}
\begin{enumerate}
    \item[]
    \item The map $f$ is measurable, $f(\cdot,x)$ is differentiable for $\mu$-almost every $x\in\X$, and $\Jac_\theta f$ is jointly continuous.

    \item There exist a non-negative locally bounded function
    $\kappa:\R^p\to\R_+$, measurable function
    $\psi_\X:\X\to\R_+$ and  $\eta>0$ such that $\int_\X \psi_\X^{2(1+\eta)}\,\di\mu<\infty,
    $ $\int_{\R^d} \|y\|^{2(1+\eta)}\,\di\nu <\infty,$ and, for all $(\theta,x)\in\R^p\times\X$, $\|f(\theta,x)\|
    +
    \|\Jac_\theta f(\theta,x)\|
    \leq
    \kappa(\theta)\psi_\X(x).$
\end{enumerate}
\end{assumption}
Under these assumptions, the empirical oracle
\eqref{eq:gradient_oracle_sliced} is a subgradient of the empirical sliced Wasserstein objective, and it converges graphically to the population subdifferential.
\begin{corollary}
\label{cor:limit_subdiff_sliced_wasserstein}
Let $\Theta\subset\R^p$ be a compact subset. Under
\Cref{ass:differentiability_sliced}, for all $n,k\in\N$, $G_{n,k}^{\SW}
\in
\partial \SW_{n,k}.$ Furthermore, almost surely, for any $\varepsilon >0$, there exists $N_\varepsilon \geq 0$ such that for any $n \geq N_\varepsilon$, there exists $K_{n} \geq 0$ such that for any $k \geq K_{n}$,
\begin{enumerate}
    \item $\partial \SW_{n,k}(\theta)
    \subset
    \partial \SW(\theta+\varepsilon\Bar{\B})
    +
    \varepsilon\Bar{\B}$ $
    \text{for all }\theta\in\Theta.$
    \item $\crit_\Theta \SW_{n,k}
    \subset
    \crit_\Theta \SW
    +
    \varepsilon\Bar{\B}.$
\end{enumerate}
\end{corollary}

\section{Graphical convergence of integrals against subsets of measures}
\label{sec:graphical_convergence}

Throughout this section, let \( S \) be a \emph{Polish space}\footnote{that is, a complete and separable metric space}. For a set-valued map \( H : \Theta \rightrightarrows \R^p \), its \emph{graph} is
\begin{equation}
\label{eq:graph}
    \Graph H := \{ (\theta,h) \in \Theta \times \R^p : h \in H(\theta) \}.
\end{equation}

We study the graphical convergence of set-valued maps defined by integration. For subsets \( P_k, P \subset \mathcal{P}(S) \), define
\[
H_k := \Bigl\{ \int_S h(\cdot,s)\,\di\gamma(s) : \gamma \in P_k \Bigr\} \qquad \text{and} \qquad
H := \Bigl\{ \int_S h(\cdot,s)\,\di\gamma(s) : \gamma \in P \Bigr\}.
\]
Under suitable assumptions, we prove that \( (H_k)_{k\in\N} \) converges graphically to \( H \) in excess, meaning that \( \Graph H_k \) is eventually contained in arbitrarily small neighborhoods of \( \Graph H \). These results are used to prove the approximation property in \Cref{th:limit_subdiff}. Background material from set-valued analysis is collected in \Cref{subsection:set_valued}, and the main convergence results are presented in \Cref{subsection:limit_integrals}.

\subsection{Preliminaries on set-valued maps}
\label{subsection:set_valued}
Unlike single-valued functions, set-valued mappings such as subdifferential \eqref{eq:clarke_subdifferential} exhibit an inherent discontinuity, which gives rise to distinct notions of semicontinuity. We expose these notions below.

\begin{definition}[Semicontinuity for set valued map] Let $Z : U \rightrightarrows V$ be a compact and nonempty-valued map.
\begin{enumerate}
    \item $Z$ is called outer semicontinuous at $u \in U$ if for any converging sequence $(u_k,v_k) \to (u,v)$ where $v_k \in Z(u_k)$ for all $k \in \N$ and  $v \in V$, one has $v \in Z(u)$. In this case, we also say that $Z$ is graph-closed.
    \item  $Z$ is called inner semicontinuous at $u \in U$ if for any $v \in Z(u)$, and for any sequence $u_k \to u$, there exists a sequence $v_k \in Z(u_k)$ such that $v_k \to v$.
    \item $Z$ is called continuous if it is both outer and inner semicontinuous.
\end{enumerate}
    
\end{definition}


Let us review these definitions, since they will be used repeatedly in the proofs. First, outer semicontinuity can be expressed as convergence with respect to the excess distance $\D$ \eqref{eq:excess_distance}. If $Z : \R^p \rightrightarrows \R^p$ is outer semicontinuous, then for any $u_k \to u$, $\D\!\bigl(Z(u_k),\, Z(u)\bigr) \xrightarrow[k \to \infty]{} 0$. Equivalently, for every $\varepsilon > 0$ there exists $\delta > 0$ such that $\|u' - u\| \le \eta$ implies $Z(u') \subset Z(u) + \varepsilon \Bar{\B}.$

Let \(\Theta_e := \Theta + e \Bar{\B}\) where $\Theta \subset \R^p$. If \(0 \le \varepsilon \le e\) and the graphical excess distance satisfies
\[
\D\!\bigl(\Graph_{\Theta_e} Z',\, \Graph_{\Theta_e} Z\bigr) \le \sqrt{2}\varepsilon,
\]
then $Z'(\theta) \subset Z(\theta + \varepsilon \Bar{\B}) + \varepsilon \Bar{\B}$ for all $\theta \in \Theta.$ This inclusion is the approximation property used in our main result \Cref{th:limit_subdiff}. Conversely, such a pointwise approximation implies a bound on the graphical excess distance.

In our analysis, the coupling set-valued map is continuous with respect to the marginal distributions, for the weak convergence (or convergence in law).
\begin{proposition}[{\cite[Th 2.3]{ghossoub2021continuity}}]  \label{prop:continuity_couplings} The set-valued map $(\mu,\nu) \rightrightarrows \Pi(\mu,\nu)$ is continuous for the weak convergence of measures.    
\end{proposition}
Particular set-valued maps used in this work include parameterized minimizers and the set of couplings. The following describes semicontinuity properties of parameterized minimizers.

\begin{proposition}[Berge maximum theorem {\cite[17.31 Theorem]{infinite}}] \label{prop:max_theorem} Let $U$ and $V$ be two topological spaces. Let $Z : U \rightrightarrows V$ be a continuous compact-valued map and let $f : \Graph Z  \to \R$ be continuous.  Then 
\begin{enumerate}
    \item $u \mapsto \max_{\gamma \in Z(u)} f(u,\gamma)$ is continuous on $U$.
    \item $u \rightrightarrows \argmax_{\gamma \in Z(u)} f(u,\gamma)$ is outer semicontinuous on $U$.
\end{enumerate}
\end{proposition}

Indeed, the maximum theorem holds replacing  (arg)max by (arg)min. Furthermore, we may apply it locally and sequentially, by considering the restrictions of maps to sequences and their closure.

Graph \eqref{eq:graph} convergence in excess admits a ``pointwise'' characterization.

\begin{lemma}[{adapted from \cite[Prop. 6.2]{aubin1987graphical}}]
\label{lem:graph_convergence_pointwise_characterization}
Let $\Theta\subset \mathbb{R}^p$ be compact.  
For each $k \in \N$, let $H_k:\mathbb{R}^p \rightrightarrows \mathbb{R}^q$ be outer semicontinuous and compact valued on $\Theta$, and let
$H:\mathbb{R}^p\rightrightarrows \mathbb{R}^q$ be outer--semicontinuous with compact values on $\Theta$. Then the following are equivalent:
\begin{enumerate}
\item $\D(\Graph_\Theta H_k,\ \Graph_\Theta H) \xrightarrow[k \to \infty]{}0$;
\item For any $\theta\in\Theta$ and any sequence $\theta_k \xrightarrow[k \to \infty]{}\theta$, $\D\bigl(H_k(\theta_k),\,H(\theta)\bigr) \xrightarrow[k \to \infty]{} 0.$
\end{enumerate}
\end{lemma}

We will also use the following stability property of graphical convergence under sum, and the stability of zero sets.
\begin{lemma}[Graphical convergence of a sum] \label{lem:graphical_cv_sum} Let $\Theta \subset \R^p$ be compact. Let $F : \R^p \rightrightarrows \R^q$ and  $G : \R^p \rightrightarrows \R^q$ be graph closed and locally bounded.  Let $(F_n)_{n \in \N}$ and $(G_n)_{n \in \N}$ be sequences of set-valued maps such that 
\begin{equation*}
    \D(\Graph_{\Theta} F_n, \Graph F) \xrightarrow[n \to \infty]{} 0 \text{ and }  \D(\Graph_{\Theta} G_n , \Graph G) \xrightarrow[n \to \infty]{} 0.
\end{equation*}
Then $\D(\Graph_\Theta (F_n + G_n), \Graph_\Theta (F + G)) \xrightarrow[n \to \infty]{} 0$.
\end{lemma}

\begin{proof}
Since \(F\) and \(G\) are graph closed and locally bounded, \(\Graph_\Theta F\) and \(\Graph_\Theta G\) are compact. Let \((\theta_n,h_n)\in\Graph_\Theta(F_n+G_n)\) with \((\theta_n,h_n)\to(\theta,h)\), \(\theta\in\Theta\).
Then \(h_n=f_n+g_n\) for some \(f_n\in F_n(\theta_n)\), \(g_n\in G_n(\theta_n)\).
By compactness, up to subsequences \(f_n\to f\) and \(g_n\to g\).
Graph closedness implies \((\theta,f)\in\Graph_\Theta F\) and \((\theta,g)\in\Graph_\Theta G\), hence
\(h=f+g\in(F+G)(\theta)\). Thus \( \D \left(\Graph_\Theta(F_n+G_n),\Graph_\Theta(F+G)\right) \xrightarrow[n \to \infty]{} 0\).
\end{proof}

\begin{lemma}[{\cite[Th. 5.37]{Rockafellar_1998}}]\label{lem:critical_set_outer}
Let $\Theta\subset\R^p$ be compact and let $F:\R^p\rightrightarrows\R^q$ have a closed graph.
Let $(F_n)_{n\in\N}$ be such that $\D(\Graph_\Theta F_n,\Graph_\Theta F)\xrightarrow[n\to\infty]{}0.$
Set $F^{-1}(0)=\{\theta\in\Theta:\ 0\in F(\theta)\}$ and
$F_n^{-1}(0)=\{\theta\in\Theta:\ 0\in F_n(\theta)\}$.
Then $\D\big(F_n^{-1}(0),\,F^{-1}(0)\big)\xrightarrow[n\to\infty]{}0.$
\end{lemma}

We state a Jensen-type inequality for the excess distance. The corresponding formula is classical in the discrete (finite-sum) setting; see, for instance, \cite{shapiro}. We extend it below to the case of an arbitrary measure.
\begin{lemma}[Jensen for excess distance] \label{lem:interchange_excess_dist_integral} Let $(S, \mathcal{A}, \mu)$ be a measure space. Let $F : S \rightrightarrows \R^p$ and $G : S \rightrightarrows \R^p$ be measurable and convex compact and non-empty valued. Then
\begin{equation*}
    \D \left(\int_S F(s) \di \mu(s), \int_S G(s) \di \mu(s) \right) \leq \int_S \D(F(s), G(s))\di \mu(s)  
\end{equation*}
\end{lemma}
\begin{proof} We use the support function representation. Let $h_A(u) := \max_{a \in A} \langle a,u\rangle$ for any convex compact set $A$. Then we have
\begin{equation*}
    \D \left( \int_S F \di \mu , \int_S G \di \mu \right)  = \sup_{\|u\| \leq 1} \max ( h_{\int_S F \di \mu}(u) - h_{\int_S G \di \mu}(u), 0).
\end{equation*}
By the maximum selection theorem \cite[18.19]{infinite} we may find a measurable selection $f^*(s) \in \argmax \langle F(s), u \rangle$, hence  we can easily verify that $h_{\int_S F \di \mu}(u) = \int_S h_{F(s)}(u)\di \mu(s)$. Thus, using the same property on $G$, we can write
\begin{align*}
     \sup_{\|u\| \leq 1} \max ( h_{\int_S F \di \mu}(u) - h_{\int_S G \di \mu}(u), 0) & \leq \sup_{\|u\|\leq 1} \max \left(\int_S \left[h_{F(s)}(u) - h_{G(s)}(u) \right]\di \mu , 0 \right) \\
     & \leq \sup_{\|u\| \leq 1} \int_S \max (h_{F(s)}(u) - h_{G(s)}(u), 0) \di \mu(s) \\   
     & \leq \int_S \D(F(s), G(s)) \di \mu(s).
\end{align*}
where the second and last inequalities are obtained by Jensen inequality.
\end{proof}


    

\subsection{Limit of integrals against subsets of measures}
\label{subsection:limit_integrals}
We now expose continuity properties of integrals $\int_S h \,\di\gamma$ with respect to the measure \(\gamma\). The results accommodate parameterized integrands
\(h : \Theta \times S \to \R^p\) and do not require boundedness on the sample space.
We first treat the single-valued case in \Cref{lem:semi_continuity_singleton}, which
constitutes a minor adaptation of elementary results (see, e.g.,
\cite[Lemma 5.1.7]{ambrosio2005gradient}) and allows us to handle discontinuous
integrands, such as those arising in spectral risk measures. We then extend the
analysis to set-valued integrands in \Cref{lem:semicontinuity_set_integrals}.
Finally, these results are combined to give
\Cref{prop:graph_convergence_parameterized_integrals} to establish graphical
convergence of the corresponding parameterized integral mappings.

\begin{lemma}[Continuity on uniformly integrable measures]
\label{lem:semi_continuity_singleton}
Let $P \subset \mathcal{P}(S)$ and $\gamma \in P$. 
Let $h : S \to \R^p$ be a measurable function.
Assume the following:
\begin{enumerate}
    \item There exists a measurable set $E \subset S$ with $\gamma(E) = 1$ such that the restriction $h$ is continuous on $E$.
    \item There exists a measurable function $\psi : S \to [0,\infty)$ such that 
    $\|h(s)\| \leq \psi(s)$ for all $s \in S$ and there exists $\eta >0$ such that $\sup_{\gamma' \in P} \int_S \psi^{1+ \eta}\,\mathrm{d}\gamma' < \infty.$
\end{enumerate}
Then for any sequence $(\gamma_k)_{k \in \N}$ from $P$ converging weakly to $\gamma$, 
$\int_S h \,\mathrm{d} \gamma_{k} \xrightarrow[k \to \infty]{} \int_S h \,\mathrm{d} \gamma.$
\end{lemma}

\begin{proof}
    We prove the result for $h : S \to \R$; the multivariate case follows by applying it componentwise. Let $(\gamma_k)_{k \in \N}$ be a sequence from $P$ converging weakly to $\gamma$.
For each $M > 0$, consider the continuous and thresholded function $h_M := \max\{-M,\min(h,M)\}.$ The only discontinuities of $h_M$ lie in $E^c$, which has $\gamma$-measure zero. Hence, by Portmanteau theorem (see, e.g., \cite[Th.~3.2.10]{durrett2010probability}), 

\begin{equation}
    \label{eq:limit_h_M_k}
    \int_S h_M \,\di\gamma_k 
    \xrightarrow[k \to \infty]{} 
    \int_S h_M \,\di \gamma.
\end{equation}
Note that 
\begin{equation}
    \label{eq:h_m_and_psi}
    |h-h_M| = \max \{|h|-M, 0\} \le |h|\,\mathds{1}_{\{|h|>M\}} \leq \psi \mathds{1}_{\{\psi > M\}}.
\end{equation}
Condition 2 and  De la Vallée Poussin criterion, see e.g. \cite[Th.~22, p.~19]{dellacherie2011probabilities}, gives uniform integrability $\lim_{M \to \infty}\sup_{\gamma' \in P} \int_S \psi(s)\,\mathds{1}_{\{\psi > M\}}(s)\di\gamma'(s) = 0$. 

Now take an arbitrary $\varepsilon > 0$. We may fix $M > 0$ such that 
$|\int_S h_M \,\di\gamma_k -
    \int_S h_M \,\di \gamma| \leq \varepsilon$ and $\sup_{\gamma' \in P} \int_S \psi(s)\,\mathds{1}_{\{\psi > M\}}(s)\di\gamma'(s) \leq \varepsilon$. Thus the limit \eqref{eq:limit_h_M_k} combined with \eqref{eq:h_m_and_psi} and uniform integrability gives $K \geq 0$ such that for all $k \geq K$,
\begin{align*}
\left|\int_S h\,\mathrm{d}\gamma_k - \int_S h\,\mathrm{d}\gamma\right|
&\le \left|\int_S h_M\,\mathrm{d}\gamma_k - \int_S h_M\,\mathrm{d}\gamma\right|
    + \int_S |h-h_M|\,\mathrm{d}\gamma_k
    + \int_S |h-h_M|\,\mathrm{d}\gamma \\
&\le 3 \varepsilon.
\end{align*}\end{proof}

\begin{lemma}
\label{lem:semicontinuity_set_integrals}
Let \( h : S \to \mathbb{R}^p \) be a measurable function, and let \( \{P_k\}_{k \ge 1}, P \subset \mathcal{P}(S) \) be weakly compact sets such that for any converging sequence $\gamma_k \in P_k$, $\gamma_k \to \gamma$, we have $\gamma \in P$. Assume:
\begin{enumerate}
    \item $h$ is continuous on a measurable subset $E \subset S$ which has full \( \gamma \)-measure for every \( \gamma \in P. \)
    \item There exists a measurable function \( \psi : S \to [0,\infty) \) such that $\|h(s)\| \leq \psi(s)$ for all $s \in S$ and  there exists $\eta >0$ such that $\sup_{\gamma' \in (\cup_{k \ge 1} P_k) \cup P} \int_S \psi^{1 + \eta} \, d\gamma' < \infty$.
\end{enumerate}

For all $k \geq 0$, if $H_k := \left\{ \int_S h \, d\gamma : \gamma \in P_k \right\}$ and $
H := \left\{ \int_S h \, d\gamma : \gamma \in P \right\}$, then $$\mathbf{D}(H_k, H) \xrightarrow[k \to \infty]{} 0.$$
\end{lemma}
\begin{proof}
Remark that by assumption, $Q := \left(\bigcup_{k \geq 0}  P_k \right)\cup P$ is sequentially compact for the weak convergence. Hence we may show that for any convergent sequence $u_k \to u$,  such that for all $k \in \N$, $u_k \in H_k$, we have $u \in H$. For such a sequence $(u_k)_{k \in \N}$, for each $k \in \N$, there exists $\gamma_k \in P_k$ satisfying $u_k = \int_S h \di \gamma_k$. Up to a subsequence, we may assume $(\gamma_k)_{k \in \N}$ weakly converges to some $\gamma \in P$. Hence the previous \Cref{lem:semi_continuity_singleton} applied to $Q$ gives the limit $\int_S h \di \gamma_k \xrightarrow[k \to \infty]{}\int_S h \di \gamma$, hence $u = \int_S h \di \gamma \in H$.
\end{proof}

\bigskip

With these lemmas in place, we now prove the graphical convergence result.

\begin{proposition}[Graph convergence for parameterized set of integrals]
\label{prop:graph_convergence_parameterized_integrals}
Let $\Theta \subset \R^p$ be a compact subset, and let $h : \Theta \times S \to \R^p$ be measurable. Let $\{P_k\}_{k \ge 1}, P \subset \mathcal{P}(S)$ be compact sets (for the weak convergence) such that for any converging sequence $\gamma_k \in P_k$, $\gamma_k \to \gamma$, we have $\gamma \in P$. Assume the following.
\begin{enumerate}
    \item There exists a set $E \subset S$ of full $\gamma$-measure for each $\gamma \in P$, such that $h$ is continuous on $\Theta \times E$.
    \item There exists a measurable function $\psi : S \to \R_+$ such that $\|h(\theta, s)\| \le \psi(s)$ for all $(\theta, s) \in \Theta \times S$,  and  there exists $\eta >0$ such that $\sup_{\gamma' \in (\cup_{k \ge 1} P_k) \cup P} \int_S \psi^{1 + \eta} \, d\gamma' < \infty$.
\end{enumerate}

For all $k \geq 0$, define $$H_k(\theta) := \left\{ \int_S h(\theta, s)\, d\gamma(s) : \gamma \in P_k \right\} \qquad \text{and}\qquad H(\theta) := \left\{ \int_S h(\theta, s)\, d\gamma(s) : \gamma \in P \right\}.$$ Then 
$\mathbf{D}(\Graph_\Theta H_k, \Graph_\Theta H) \xrightarrow[k \to \infty]{} 0$.
\end{proposition}

\begin{proof}
    In order to apply \Cref{lem:graph_convergence_pointwise_characterization} let $\theta \in \Theta$ and $\theta_k \xrightarrow[]{} \theta$. Consider the subsets of joint measures
    \begin{equation*}
        Q_k := \left\{ \delta_{\theta_k} \otimes \gamma_k \ : \gamma_k \in P_k \right\}.
    \end{equation*}

    Let $q_k \in Q_k$ be a converging sequence to some $q$. This implies convergence of the marginals of $q_k$. The second marginal of $q_k$ is a measure $\gamma_k \in P_k$ hence $\gamma_k$ weakly converges to some  $\gamma \in P$ by assumption. Furthermore, $\delta_{\theta_k} \to \delta_\theta$  in probability. Thus, by Slutsky theorem, $q_k$ weakly converges to $\delta_\theta \otimes \gamma$. This means, if $Q :=\{\delta_\theta \otimes \gamma \ : \ \gamma \in P\}$, any converging sequence $q_k \in Q_k$ converges in $Q$. Hence if we define
    \begin{equation*}
        G_k :=  \left\{\int_{\Theta \times S} h \di q_k \ : \ q_k \in Q_k \right\} \qquad G :=  \left\{\int_{\Theta \times S} h \di q \ : \ q \in Q \right\}
    \end{equation*}
    then all conditions are satisfied to apply \Cref{lem:semicontinuity_set_integrals} to the set-valued maps $(G_k)_{k \in \N}$ and $G$ with the refined subsets $(Q_k)_{k \in \N}$ and $Q$. Thus we have $\D(G_{k}, G) \xrightarrow[k \to \infty]{} 0$. By the definition of $Q_k$ and $Q$, this also writes $\D(H_{k}(\theta_k), H(\theta)) \xrightarrow[k \to \infty]{} 0.$
    This holds for any $\theta \in \Theta$ and $\theta_k \to \theta$, thus we may apply \Cref{lem:graph_convergence_pointwise_characterization} to conclude.
\end{proof}

\section{Subdifferential of parameterized transport costs}
\label{sec:4_convergence_subdiff_OT}

In this section, we prove the graphical convergence of empirical subgradients, as stated in our main result \Cref{th:limit_subdiff}.
The key observation is that, for any $\theta$, the population transport cost \eqref{eq:transport_cost_population} can be reformulated as an optimal transport problem between the sample spaces $\X$ and $\Y$. This relies on the existence of a \emph{lift}, which connects transport plans between pushforward measures to plans between the underlying data distributions \cite[Lemma~2.6]{dumont2025existence}; see also \cite[\S 7]{lift_tanguy2025sliced}. For convenience, we use the following consequence.

\begin{lemma}[{adapted from \cite[Lemma~2.7]{dumont2025existence}}]
\label{lem:lift_transport}
Let $c : \mathcal{U} \times \mathcal{V} \to \R_+$, $f : \X \to \mathcal{U}$ and $g : \Y \to \V$ be measurable. Then
\begin{equation*}
   \inf_{\pi \in \Pi(f_\#\mu,\, g_\# \nu)}
   \int_{\mcU \times \V} c(u, v)\, \di\pi(u,v)
   \;=\;
   \inf_{\gamma \in \Pi(\mu, \nu)}
   \int_{\X \times \Y} c(f(x), g(y))\, \di\gamma(x,y).
\end{equation*}
\end{lemma}

In particular, \eqref{eq:transport_cost_population} admits the equivalent formulation
\begin{equation*}
   \T_c(\theta)
   = \inf_{\gamma \in \Pi(\mu, \nu)}
     \int_{\X \times \Y} c\bigl(f(\theta,x),\, g(\theta,y)\bigr)\, \di\gamma(x,y).
\end{equation*}
By introducing  $C(\theta,x,y) := c\bigl(f(\theta,x),\, g(\theta,y)\bigr)$, this naturally leads us to study the parameterized optimal transport function
\begin{equation}
    \label{eq:parameterized_transport_cost}
    T_C(\theta, \mu, \nu)
    := \inf_{\gamma \in \Pi(\mu,\nu)}
       \int_{\X \times \Y} C(\theta, x, y)\, d\gamma(x,y).
\end{equation}

In the case of \eqref{eq:parameterized_transport_cost}, we call \emph{optimal transport plan} any $\gamma \in \Pi(\mu,\nu)$ attaining the infimum.

In \Cref{subsection:parametric_stability}, we analyze the optimal transport problem \eqref{eq:parameterized_transport_cost} and prove parametric stability results for $T_C$. In \Cref{subsection:envelope_transport_cost} we present an envelope formula for the subgradients of $T_C$. In \Cref{subsection:proof_convergence_empirical_subdiff}, we combine the two previous subsections to obtain our convergence result from \Cref{subsec:main_results_vanilla_transport}. Extension to sliced costs, \Cref{subsec:main_results_vanilla_sliced}, is established in \Cref{subsec:sliced_costs_proofs}.

\subsection{Parametric stability of transport costs}
\label{subsection:parametric_stability}

Below, we establish the qualitative stability of the transport problem  \eqref{eq:parameterized_transport_cost} jointly with respect to the parameter $\theta$ and the marginals $\mu,\nu$. This is mainly an application of the maximum theorem, recalled in \Cref{prop:max_theorem}. We work under the following assumptions
\begin{assumption}
    \label{assumption:regularity_cost}
    \begin{enumerate}
    \item[]
    \item (Continuity off null sets) There exist full-measure open sets $E \subset \X$ and $F \subset \Y$ such that $\mu(E) = \nu(F) = 1$ and $C$ is continuous on $\R^p \times E \times F$.
    \item (Uniform integrability)  
    There exists a locally bounded function $\kappa : \R^p \to \R_+$ and measurable functions $\psi_\X : \X \to  \R_+$, $\psi_\Y : \Y \to  \R_+$ such that $|C(\theta,x,y)| \le \kappa(\theta)( \psi_\X(x) + \psi_\Y(y))$ for all $(\theta,x,y) \in \R^p \times  \X \times \Y$.
\end{enumerate}
\end{assumption}

The stability result then states as follows.
\begin{proposition}[Parametric stability of optimal transport] 
\label{prop:parametric_stability}
Under \Cref{assumption:regularity_cost}, let $T_C$ be defined as \eqref{eq:parameterized_transport_cost}. Let $\mu_n \xrightarrow[n \to \infty]{} \mu \in \mathcal{P}(\X)$ and $\nu_m \xrightarrow[m \to \infty]{} \nu \in \mathcal{P}(\Y)$ weakly, and $\theta_k \xrightarrow[k \to \infty]{} \theta \in \R^p$.  Assume there exists $\eta > 0$ such that $\sup_{n,m \geq 0} \int_\X \psi_\X^{1+\eta} \di \mu_n + \int_\Y \psi_\Y^{1+ \eta} \di \nu_m < \infty$. Then $T_C(\theta_k, \mu_n,\nu_m) \xrightarrow[\min(k,n,m) \to \infty]{} T_C(\theta, \mu,  \nu)$. Furthermore, any weak accumulation point of the optimal transport plans for $T_C(\theta_k, \mu_n,\nu_m)$  as $\min(k,n,m) \to \infty$, is optimal for $T_C(\theta, \mu,\nu)$.    
\end{proposition}

\begin{proof}
 We will apply the maximum theorem locally. It is thus sufficient to verify continuity conditions at $\theta, \mu, \nu$ and $\gamma \in \Pi(\mu,\nu)$. First, continuity of the coupling map $Z : (\mu,\nu) \rightrightarrows \Pi(\mu,\nu)$ is given by \Cref{prop:continuity_couplings}. Now, let us verify that under condition 1 and 2, $f : (\theta', \mu',\nu', \gamma) \mapsto \int_{\X \times \Y} C(\theta', x,y) \di \gamma(x,y)$ is continuous at $(\theta, \mu,\nu,\gamma)$ on the set $$\{(\theta',\mu', \nu', \gamma') \ : \ \theta' \in \Theta, \mu' \in \mathcal{P}(\X), \nu' \in\mathcal{P}(\Y), \gamma' \in \Pi(\mu',\nu') \}.$$

Then, take arbitrary and weakly converging sequences $\theta_k' \to \theta$  and   $\mu_k' \to \mu$, $\nu_k' \to \nu$ and $\gamma_{k}' \to \gamma$. Define $P_k := \{\gamma_k'\}$ and $P := \{\gamma\}$. Let $\Psi(x,y) := \psi_\X(x) + \psi_\Y(y)$.    By continuity of $\psi_\X, \psi_\Y$ and since they are lower bounded, we have by assumption $\int_{\X} \psi_\X(x)^{1 + \eta} \di \mu(x) + \int_\Y\psi_\Y(y)^{1+\eta} \di \nu(y)< \infty$. Also,  remark that $\Psi(x,y)^{1 + \eta} \leq 2^{\eta} (\psi_\X(x)^{1 + \eta} + \psi_\Y(y)^{1+\eta})$ for all $x,y$ hence we have $\sup_{\gamma' \in \left(\bigcup_{k \geq 0} P_k \right) \cup P} \int_{\X \times \Y} \Psi^{1 + \eta} \di \gamma' < \infty$. Furthermore,  $C$ is jointly continuous on  $\R^p \times E \times F$ where $E \times F$ has full measure with respect to $\gamma$, and $\|C(\theta, x,y)\| \leq \kappa(\theta)\Psi(x,y)$ by \Cref{assumption:regularity_cost}. All the conditions are satisfied to apply \Cref{prop:graph_convergence_parameterized_integrals} to obtain $\int_{\X \times \Y} C(\theta_k, x,y) \di \gamma_k(x,y) \xrightarrow[k \to \infty]{} \int_{\X \times \Y} C(\theta, x,y) \di \gamma(x,y).$ Continuity of $f$ is verified. All the conditions are satisfied to apply the maximum theorem, \Cref{prop:max_theorem}, hence the result.
\end{proof}

Observe that the proposition above can be viewed as a parametric analogue of the classical qualitative stability result for transport costs. In contrast to the standard argument in \cite[Theorem 5.18]{villani}, which relies on cyclical monotonicity and considers uniformly converging cost functions, we propose an alternative perspective by using general set-valued analysis arguments. As an illustration, we obtain the following consequence for empirical distributions.

\begin{corollary}
Under \Cref{assumption:regularity_cost}, let \((x_n)_{n \in \N}\) and \((y_m)_{m \in \N}\) be i.i.d.\ sequences with laws
\(\mu\) and \(\nu\), respectively, and denote the associated empirical measures by
\[
\mu_n := \frac{1}{n} \sum_{i=1}^n \delta_{x_i},
\qquad
\nu_m := \frac{1}{m} \sum_{j=1}^m \delta_{y_j}.
\]
Then, almost surely with respect to the sequences
\((x_n)_{n \in \N}\) and \((y_m)_{m \in \N}\), the following holds: for any sequence \(\theta_k \to \theta\), $T_C(\theta_k,\mu_n,\nu_m)
\xrightarrow[\min(k,n,m) \to \infty]{}
T_C(\theta,\mu,\nu),$
and any weak accumulation point of any sequence of optimal transport plans
\(\gamma_{k,n,m}\) as $\min(k,n,m) \to \infty$, is optimal for \(T_C(\theta,\mu,\nu)\).
\end{corollary}
\begin{proof}
     By law of large number, the condition $\sup_{n,m}  \int_\X \psi_{\X}^{1+\eta} \di \mu_n +   \int_\Y \psi_{\Y}^{1+\eta} \di \nu_m < \infty$ is satisfied almost surely. The result is then a direct application of \Cref{prop:parametric_stability}.
\end{proof}

\subsection{Envelope formula  for parameterized transport costs}
\label{subsection:envelope_formula}

We now formulate an envelope theorem suited to parameterized transport costs, and clarify how it relates to our integrability assumptions. For convenience, we first recall a general form of Danskin theorem \cite{danskin2012theory}, due to Clarke \cite{clarke1975generalized}.
\label{subsection:envelope_transport_cost}
\begin{proposition}[Envelope formula, adapted from {\cite[Th. 2.8.2]{clarke1990optimization}}]
\label{prop:envelope}
Let $\Theta \subset \R^p$ be open. Define $F : \Theta \to \R$ by $F(\theta) = \max_{\gamma \in V} f(\theta,\gamma)$ for each $\theta \in \Theta$. Assume the following:
\begin{enumerate}
    \item $V$ is a separable and sequentially compact space.
    \item $\gamma \mapsto f(\theta, \gamma)$ is continuous on $V$ for each $\theta \in \Theta$.
    \item $(\theta, \gamma) \mapsto \nabla_\theta f(\theta, \gamma)$ is continuous on $\Theta \times V$.    
    \item Each $f(\cdot, \gamma)$, $\gamma \in V$, is Lipschitz on $\Theta$ with common constant $L$, and for each $\theta \in \Theta$, $\{f(\theta, \gamma) : \gamma \in V\}$ is bounded.
\end{enumerate}
Then $F$ is Lipschitz on $\Theta$ with constant $L$ and 
$$
\partial F(\theta)
=
\overline{\conv} \left\{ 
\nabla_{\theta} f(\theta,\gamma) 
: \gamma \in \argmax_{\gamma' \in V} f(\theta, \gamma')
\right\}.
$$
\end{proposition}

We will apply this envelope formula either for maximum functions (e.g., spectral risk \Cref{subsection:spectral_risk}) or for minimum functions (standard transport cost). For a minimum function defined as $F(\theta) = \min_{\gamma \in \Gamma} f(\theta, \gamma)$, since $\partial (-F) = - \partial F$, \Cref{prop:envelope} applied to $-F$ and $-f$ gives
\begin{equation*}
     \partial F(\theta) = \overline{\conv} \left\{ \nabla_\theta f(\theta,\gamma) \ : \ \gamma \in \argmin_{\gamma' \in V} f(\theta, \gamma')\right\} 
\end{equation*}

The formula specializes to our transport setting from \Cref{sec:main_results}.


\begin{proposition}[Envelope formula for parameterized transport]
\label{prop:envelope_ot}
Let $\Theta' \subset \R^p$ be open and bounded.  Let $C$ satisfy \Cref{ass:differentiability_general_transport_cost} and let $T_C$ be defined as \eqref{eq:parameterized_transport_cost}. Then $T_C(\cdot, \mu,\nu)$ is Lipschitz continuous on $\Theta'$ with constant $\sup_{\theta \in \Theta'}\kappa(\theta)(\int_{\X} \psi_\X \, d\mu + \int_{\Y} \psi_\Y \, d\nu)$, and the envelope formula holds: for any $\theta \in \Theta'$,
\[
\partial_\theta T_C(\theta,\mu,\nu)
=
\left\{
\int_{\X \times \Y}  \nabla_\theta C(\theta,x,y) \, \di\gamma(x,y)
\;:\;
\gamma \in  \argmin_{\gamma' \in \Pi(\mu,\nu)} \int_{\X \times \Y} C(\theta,x,y) \, \di\gamma'(x,y) \right\}.
\]
\end{proposition}

\begin{proof}
Our goal is to apply the envelope theorem (Proposition~\ref{prop:envelope}) with 
\[
f(\theta,\gamma) = \int_{\X \times \Y} C(\theta,x,y)\,\di\gamma(x,y) \qquad \text{and} \qquad  V = \Pi(\mu,\nu).
\]
We verify the hypotheses of Proposition~\ref{prop:envelope} in turn.

\emph{(1) $V$ is separable and sequentially compact.}
Since $\mathcal{P}(\X \times \Y)$ is Polish under the weak topology, $V$ is separable \cite[Th 6.8]{billingsley2013convergence}. Furthermore, $V = \Pi(\mu, \nu)$ is sequentially compact, see \cite[Lemma 4.3]{villani}.

\emph{(2) Continuity of $\gamma \mapsto f(\theta,\gamma)$.} Let $\Psi(x,y) := \psi_\X(x) + \psi_\Y(y)$ and $\theta \in \Theta'$. Let $(\gamma_n)_{n \in \N}$ be an arbitrary sequence from $V$ converging to $\gamma \in V$. Remark that $(\psi_\X(x) + \psi_\Y(y))^{1+\eta}\leq 2^{\eta} (\psi_\X(x)^{1+ \eta}+ \psi_\Y(y)^{1+\eta})$. Hence under \Cref{ass:differentiability_general_transport_cost}, we may apply \Cref{lem:semi_continuity_singleton} with the choice $P = V$ and $h = C(\theta, \cdot)$ to obtain $f(\theta, \gamma_n) \to f(\theta,\gamma)$.

\emph{(3) Continuity of $(\theta,\gamma) \mapsto \nabla_\theta f(\theta, \gamma)$}. 
Let $\gamma_n$ be an arbitrary sequence from $V$ converging to $\gamma \in V$, and $\theta_n \to \theta$. By assumption, $\nabla_\theta C$ is continuous on $\Theta' \times E \times F$ where $E \times F$ has full measure with respect to any $\gamma \in V$. Furthermore, $\|\nabla_\theta C(\theta,x,y) \|\leq \kappa(\theta)\Psi(x,y) = \kappa(\theta)(\psi_\X(x) + \psi_\Y(y))$  for all $\theta,x,y$. As in the proof of (2), under condition 2 of \Cref{ass:differentiability_general_transport_cost}, we easily verify that $\sup_{\gamma \in V} \int  \Psi^{1+ \eta} \di \gamma < \infty$. Hence we may apply \Cref{prop:graph_convergence_parameterized_integrals} with $P_k = P = V$ and $h = \nabla_\theta C$ to have the desired convergence $\nabla_{\theta} f(\theta_n, \gamma_n) \xrightarrow[n \to \infty]{} \nabla_\theta f(\theta, \gamma)$.

\emph{(4) Boundedness and Lipschitz constant in $\theta$.}
For any $\gamma \in V$ and $\theta,\theta' \in \Theta'$, continuous differentiability and Cauchy-Schwarz inequality yields
\begin{equation*}
    |C(\theta',x,y) - C(\theta,x,y)| \leq  \int_0^1 \| \nabla_\theta C(t\theta' + (1 - t)\theta,x,y)\|\| \theta'- \theta \| \di t
\end{equation*}
hence by integrating with respect to an arbitrary $\gamma \in V$ on both sides, we easily obtain
\begin{equation*}
    |f(\theta',\gamma) - f(\theta,\gamma)| 
\le  \sup_{\theta \in \Theta'}\kappa(\theta) \left( \int_\X  \psi_\X\di \mu+ \int_\Y  \psi_\Y\di \nu\right)\|\theta' - \theta\|
\end{equation*}
Thus, the family $f(\cdot,\gamma)$ is Lipschitz with global constant given above. Boundedness of $\{f(\theta, \gamma) \ : \ \gamma \in V\}$ is obtained from condition 2 of \Cref{ass:differentiability_general_transport_cost}.

All assumptions of Proposition~\ref{prop:envelope} are now satisfied, hence the result.
\end{proof}

\subsection{Convergence of empirical subgradients for parameterized costs.} 
\label{subsection:proof_convergence_empirical_subdiff}

We are now ready to obtain our main result. First, we establish it for parameterized cost, and then, we conclude on the transport problem between parameterized models, \eqref{eq:transport_cost_population} by using the existence of lift, \Cref{lem:lift_transport}.

\begin{proposition} \label{prop:graphical_convergence_general} Let $\Theta \subset \R^p$ be compact with nonempty interior. Let $C$ satisfy \Cref{ass:differentiability_general_transport_cost} and let $T_C$ be defined as \eqref{eq:parameterized_transport_cost}. Let $\mu_n \xrightarrow[n \to \infty]{} \mu \in \mathcal{P}(\X)$ and $\nu_m \xrightarrow[m \to \infty]{} \nu \in \mathcal{P}(\Y)$ weakly. Assume there exists $\eta > 0$ such that $\sup_{n,m \geq 0} \int_\X \psi_\X^{1+\eta} \di \mu_n + \int_\Y \psi_\Y^{1+ \eta} \di \nu_m < \infty$. Then $\D \left(\Graph_\Theta \partial_\theta T_{C}(\cdot, \mu_n,\nu_m), \Graph_\Theta \partial_\theta T_C(\cdot,\mu,\nu) \right) \xrightarrow[\min(n,m) \to \infty]{} 0.$
\end{proposition}

\begin{proof}
   Under \Cref{ass:differentiability_general_transport_cost},  \Cref{prop:envelope_ot} applies to any $n,m \in \N$ to give $\partial_\theta T_C(\theta,\mu_n,\nu_m)$ is equal to
    \[
\left\{
\int_{\X \times \Y}  \nabla_\theta C(\theta,x,y) \, \di\gamma(x,y)
\;:\;
\gamma \in  \argmin_{\gamma' \in \Pi(\mu_n,\nu_m)} \int_{\X \times \Y} C(\theta,x,y) \, \di\gamma'(x,y) \right\}.
\]

Take any subsequence $\mu_{n_k}, \nu_{m_k}$ where $\min(n_k,m_k) \xrightarrow[k \to \infty]{} \infty$. Let $\theta_k \to \theta \in \Theta$ be an arbitrary converging sequence. Define the following subsets:
\begin{equation*}
    \begin{aligned}
        & P_k := \argmin_{\gamma' \in \Pi(\mu_{n_k},\nu_{m_k})} \int_{\X \times \Y} C(\theta_k,x,y) \, \di\gamma'(x,y) \qquad \forall k \geq 0. \\ 
        & P := \argmin_{\gamma' \in \Pi(\mu,\nu)} \int_{\X \times \Y} C(\theta,x,y) \, \di\gamma'(x,y).
    \end{aligned}
\end{equation*}
With these notations, \Cref{prop:parametric_stability} gives us that any (weak) accumulation point of any sequence $\gamma_k  \in P_k$ belongs to $P$. Hence, we may apply \Cref{prop:graph_convergence_parameterized_integrals} with $h := \nabla_\theta C$ and the choices of $P_k$, $P$ above to obtain the result.
\end{proof}

This indeed applies to empirical distributions.

\begin{corollary}
\label{corollary:graphical_conv_parameterized_empirical}
    Under \Cref{ass:differentiability_general_transport_cost}, let \((x_n)_{n \in \N}\) and \((y_m)_{m \in \N}\) be i.i.d.\ sequences with laws
\(\mu\) and \(\nu\), respectively, and denote the associated empirical measures by $\mu_n := \frac{1}{n} \sum_{i=1}^n \delta_{x_i}$ and $\nu_m := \frac{1}{m} \sum_{j=1}^m \delta_{y_j}.$ Then almost surely with respect to the sequences
\((x_n)_{n \in \N}\) and \((y_m)_{m \in \N}\), \Cref{prop:graphical_convergence_general} holds.
\end{corollary}

\begin{proof}
     By law of large number, the condition $\sup_{n,m}  \int_\X \psi_{\X}^{1+\eta} \di \mu_n +   \int_\Y \psi_{\Y}^{1+\eta} \di \nu_m < \infty$ is satisfied almost surely. Furthermore, almost surely, empirical measures weakly converge to their population limit \cite[Th. 11.4.1]{dudley2018real}. The result is then a direct application of \Cref{prop:graphical_convergence_general}.
\end{proof}

\bigskip
Now we conclude and prove \Cref{th:limit_subdiff}.

\begin{proof}[of \Cref{th:limit_subdiff}] 
Under \Cref{ass:differentiability_general_transport_cost}, let $C : (\theta,x,y) \mapsto c(f(\theta,x),g(\theta,y)).$

By \Cref{lem:lift_transport}, for each $\theta \in \R^p$, $\T_c(\theta) = \min_{\gamma \in \Pi(\mu,\nu)} \int_{\X \times \Y} C(\theta,x,y) \di \gamma(x,y) = T_C(\theta,\mu,\nu)$
where the minimum is attained by the condition $|C(\theta,x,y)| \leq \kappa(\theta) \left(\int_\X \psi_\X \di \mu + \int_\Y \psi_\Y \di \nu \right)< \infty$
of \Cref{ass:differentiability_general_transport_cost}.2. We may also apply \Cref{lem:lift_transport} with the empirical distributions to obtain $\T^{n,m} = T^{n,m}_C$. With this choice of parameterized cost $C$, we may apply \Cref{corollary:graphical_conv_parameterized_empirical} to obtain $\D(\Graph_\Theta \partial \T_c^{n,m}, \Graph_\Theta \partial \T_c ) \xrightarrow[\min(n,m) \to \infty]{} 0$ almost surely.
\end{proof}
\subsection{Convergence of empirical subgradients for sliced costs}

\label{subsec:sliced_costs_proofs}

We now prove the graphical convergence result for sliced costs, \Cref{th:limit_subdiff_sliced}. $\mathcal{K} \subset \R^d$ is a compact set, and we consider a parameterized cost function $C : \Theta \times \X \times \Y \times \mathcal{K} \to \R$. Remark that compared to the setting of \Cref{subsection:parametric_stability}, we add the dependence on the projection vectors $\phi \in \mathcal{K}$. For any distributions $\mu' \in \mathcal{P}(\X)$, $\nu' \in \mathcal{P}(\Y)$ and $\rho' \in \mathcal{P}(\mathcal{K})$ we define the parameterized sliced cost as
\begin{equation}
\label{eq:sliced_cost_tC}
    S_C(\theta, \mu',\nu',\rho') = \int_{\mathcal{K}} t_C(\theta, \mu',\nu',\phi) \di \rho'(\phi)
\end{equation}
where $t_C$ is the cost $t_C(\theta, \mu',\nu',\phi):= \inf_{\gamma \in \Pi(\mu',\nu')} \int_{\X \times \Y} C(\theta,x,y,\phi) \di \gamma(x,y).$



The convergence states as follows.

\begin{proposition}  \label{prop:graph_convergence_param_sliced} Let $C$ satisfy  \Cref{assumption:main_differentiability_parameterized_sliced_cost} and let $S_C$ be defined as \eqref{eq:sliced_cost_tC} let $(\mu_n)_{n \in \N}$, $(\nu_m)_{m \in \N}$ be sequences of probability distributions such that $\mu_n \xrightarrow[n \to \infty]{} \mu$, $\nu_m \xrightarrow[m \to \infty]{} \nu$ weakly and $\sup_{n,m} \int_{\X} \psi_\X^{1+\eta} \di \mu_n + \int_\Y \psi_\Y^{1+\eta} \di \nu_m<  \infty$. 
Let $(\rho_k)_{k \in \N}$ be the sequence of empirical distributions from $\rho$, for all $k \in \N$, $\rho_k = \frac{1}{k} \sum_{j=1}^k \delta_{\phi_j}$ where $(\phi_j)_{j \in \N}$ is i.i.d from $\rho$. For all $\theta \in \Theta$, $n,m,k \in \N$, let $S_C^{n,m,k}(\theta) := S_C(\theta,\mu_n,\nu_m,\rho_k)$ and let $S_C^\infty(\theta) := S_C(\theta,\mu,\nu,\rho)$. Then almost surely, for any $\varepsilon >0$ there exists $N_\varepsilon \geq 0$ such that for all $n,m \geq N_\varepsilon$, there exists $K_{n,m}$ such that for all $k \geq K_{n,m}$,   $\D \left(\Graph_\Theta \partial S_C^{n,m,k}, \Graph_\Theta \partial S_C^\infty \right) \leq  \varepsilon.$
\end{proposition}
Remark compared to \Cref{prop:graphical_convergence_general}, we restrict to empirical i.i.d approximations of \(\rho\) for simplicity. Extending the analysis to arbitrary weakly convergent approximations \(\rho_k \rightarrow \rho\) would require intermediary results which we omit here.

\begin{proof}
Throughout the proof, we set
\[
t_C^{n,m}(\theta,\phi) := t_C(\theta,\mu_n,\nu_m,\phi),
\qquad
t_C^\infty(\theta,\phi) := t_C(\theta,\mu,\nu,\phi).
\]

\medskip

Let $(n_l)_{l\in\mathbb N}$ and $(m_l)_{l\in\mathbb N}$ be arbitrary subsequences with
$n_l\to\infty$ and $m_l\to\infty$.
We aim to apply the characterization of graphical convergence given in
\Cref{lem:graph_convergence_pointwise_characterization}. By the set-valued Jensen formula, \Cref{lem:interchange_excess_dist_integral}, we have
\begin{equation}
	\label{eq:slice_pointwise_limit}
	\begin{aligned}
		&\D\!\Biggl(
		\int_{\mathcal K}
		\partial_\theta t_C^{n_l,m_l}(\theta_l,\phi)\,\mathrm d\rho(\phi),
		\\
		&\hspace{3.2em}
		\int_{\mathcal K}
		\partial_\theta t_C^\infty(\theta,\phi)\,\mathrm d\rho(\phi)
		\Biggr)
		\\
		&\qquad\le
		\int_{\mathcal K}
		\D\!\left(
		\partial_\theta t_C^{n_l,m_l}(\theta_l,\phi),
		\partial_\theta t_C^\infty(\theta,\phi)
		\right)
		\,\mathrm d\rho(\phi).
	\end{aligned}
\end{equation}

Under \Cref{assumption:main_differentiability_parameterized_sliced_cost},
\Cref{ass:differentiability_general_transport_cost} holds for each
$C(\cdot,\cdot,\cdot,\phi)$ with $\phi\in\mathcal K$.
Hence, by combining \Cref{prop:graphical_convergence_general} and
\Cref{lem:graph_convergence_pointwise_characterization}, we obtain for every
$\phi\in\mathcal K$, $$\D\!\left(
\partial_\theta t_C^{n_l,m_l}(\theta_l,\phi),
\partial_\theta t_C^\infty(\theta,\phi)
\right)
\xrightarrow[l\to\infty]{} 0.$$
To pass to the limit in \eqref{eq:slice_pointwise_limit}, we need a uniformly integrable bound.
Under \Cref{assumption:main_differentiability_parameterized_sliced_cost}, we may apply the envelope
formula from \Cref{prop:envelope_ot}, which yields for all $\theta\in\Theta$ and $\phi\in\mathcal K$,
\[
\partial_\theta t_C^\infty(\theta,\phi)
=
\left\{
\int_{\mathcal X\times\mathcal Y}
\nabla_\theta C(\theta,x,y,\phi)\,\mathrm d\gamma(x,y)
\;:\;
\gamma\in\argmin_{\gamma'\in\Pi(\mu,\nu)}
\int C(\theta,x,y,\phi)\,\mathrm d\gamma'
\right\},
\]
and an analogous formula holds for the empirical counterpart $\partial_\theta t_C^{n,m}(\theta,\phi)$. Since the gradients $\nabla_\theta C$ are integrably bounded and the empirical measures converge
almost surely, the strong law of large numbers implies that the right-hand side of
\eqref{eq:slice_pointwise_limit} is dominated by an integrable function.
Thus, by the dominated convergence theorem, the right-hand side of
\eqref{eq:slice_pointwise_limit} converges to $0$ as $l\to\infty$.

\medskip

We now consider the convergence of the integrated set-valued maps.
By \eqref{eq:slice_pointwise_limit} and \Cref{lem:graph_convergence_pointwise_characterization}, $\int_{\mathcal K} \partial_\theta t_C^{n_l,m_l}(\cdot,\phi)\,\mathrm d\rho(\phi)$ converges graphically to $
\int_{\mathcal K} \partial_\theta t_C^\infty(\cdot,\phi)\,\mathrm d\rho(\phi)$ on $\Theta$. Let us invoke a technical point needed only for the interchange of integration and
subdifferentiation.
Since the envelope formula in \Cref{prop:envelope_ot} is derived from
\cite[Th.~2.8.2]{clarke1990optimization}, the functions
$-t_C^{n,m}$ and $-t_C^\infty$ are \emph{Clarke regular}  with respect to $\theta$, see \cite[Def. 2.3.4]{clarke1990optimization}.
Consequently, the subdifferential–integral interchange theorem
\cite[Th.~2.7.2]{clarke1990optimization} applies, yielding
\[
S_C^{n,m}
=
\int_{\mathcal K} \partial_\theta t_C^{n,m}(\cdot,\phi)\,\mathrm d\rho(\phi),
\qquad
S_C^\infty
=
\int_{\mathcal K} \partial_\theta t_C^\infty(\cdot,\phi)\,\mathrm d\rho(\phi).
\]
where these notations are given in the statement of \Cref{prop:graph_convergence_param_sliced}. Hence, for any $\varepsilon>0$, there exists $N_\varepsilon\in\mathbb N$ such that for all
$n,m\ge N_\varepsilon$, $\D\!\left(
\Graph_\Theta \partial S_C^{n,m},
\Graph_\Theta \partial S_C^\infty
\right)
\le \varepsilon/2.$
Moreover, for each fixed $(n,m)$, there exists $K_{n,m}\in\mathbb N$ such that for all
$k\ge K_{n,m}$, $\D\!\left(
\Graph_\Theta \partial S_C^{n,m,k},
\Graph_\Theta \partial S_C^{n,m}
\right)
\le \varepsilon/2$ almost surely,
by the law of large numbers for graphs of set-valued maps
\cite[Th.~4.11]{NorWet13}.
The result then follows from the triangle inequality for $\D$ and by taking a countable intersection of almost sure events.
\end{proof}

\begin{proof}[of \Cref{th:limit_subdiff_sliced}] Our result now follows easily. By the lift representation \Cref{lem:lift_transport}, each transport \eqref{eq:main_little_transport_c} writes  $$t_c(\theta,\phi) = \min_{\gamma \in \Pi(\mu_n, \nu_m) } \int_{\X \times \Y} c(\phi^\top f(\theta,x), \phi^\top g(\theta,y)) \di \gamma(x,y)$$
hence by chosing $C$ as \eqref{eq:sliced_parameterized_cost}, under \Cref{assumption:main_differentiability_parameterized_sliced_cost}, we may apply \Cref{prop:graph_convergence_param_sliced} to have the result.
\end{proof}

\subsection{Application to transport-based objectives}
\label{subsec:proof_applications}
This section gathers proofs of the applications in \Cref{section:2_applications}.

\subsubsection{Spectral risk.}

\begin{proof}[of \Cref{cor:limit_spectral_risk}]
Under \Cref{ass:differentiability_spectral_risk}.1, remark that if $Q := w_{\#} \nu$ with $\nu$ the uniform distribution $[0,1]$, then $w$ coincides Lebesgue-a.e.\ with the quantile function $F_Q^{-1}$. This allows to write the spectral risk \eqref{eq:SR_population} as a transport problem, see for instance \cite[Lemma 4.55]{follmer2011stochastic} for such a result:
\begin{equation*}
    - \SR(\theta) =  \min \left\{- \int s z \di \pi(s,z) \ : \ \pi \in \Pi( w_{\#} \nu, (\ell_\theta)_\# \mu)  \right\}.
\end{equation*}
The same formula holds for the empirical objective $\SR_n$ \eqref{eq:SR_empirical}. Indeed, set, $\nu_n := \frac{1}{n} \sum_{i=1}^n \delta_{i/n}$ and $\mu_n := \frac{1}{n}\sum_{i=1}^n \delta_{x_i}$ where the sequence $(x_i)_{i \in \N}$ is i.i.d from $\mu$. Then we have by similar arguments
\begin{equation*}
    -\SR_n(\theta) = \min \left\{- \int s z \di \pi(s,z) \ : \ \pi \in \Pi( w_{\#} \nu_n, (\ell_\theta)_\# \mu_n)  \right\}.
\end{equation*}

The formula of $G^{\SR}$ is an application of the existence of lift \Cref{lem:lift_transport} and the envelope formula \Cref{prop:envelope_ot} under \Cref{ass:differentiability_spectral_risk} with $C(\theta,x, y) := w(y) \ell(\theta,x)$. Under  \Cref{ass:differentiability_spectral_risk}, we have
\begin{align*}
    |C(\theta,x,y)|+ \|\nabla_\theta C(\theta,x,y)\| & \leq  \|w\|_\infty \left(|\ell(\theta,x)| + \|\nabla_\theta \ell(\theta,x)\| \right) \\& \leq  \|w\|_\infty \kappa(\theta) \psi(x)
\end{align*}
and $\int_\X \psi^{1+ \eta} \di \mu$ and $\|w\|_\infty$ are finite. Note that $C$ is continuous almost everywhere in $x$ and $y$ by \Cref{ass:differentiability_spectral_risk}.2.  

By law of large numbers, almost surely $\sup_{n \geq 0} \int_\X \psi^{1+ \eta}  \di \mu_n < \infty$. Furthermore if $\nu_n:=\frac{1}{n} \sum_{i=1}^n{\delta_{i/n}}$, we have $\nu_n \xrightarrow[]{} \nu$ weakly (by convergence of Riemann sums). Thus, we may apply \Cref{th:limit_subdiff} to obtain the result.
\end{proof}

\subsubsection{Fair learning.} We go to our result on fair learning, \Cref{cor:limit_fairness}. First, we prove the formula for unbalanced sample size \Cref{lem:unbalanced_wasserstein}, with further details on the associated transport plan.

\begin{lemma}[Wasserstein cost for unequal sample sizes]
\label{lem:unbalanced_wasserstein_extended}
Let $\mu_U = \frac1n \sum_{i=1}^n \delta_{u_i}$ and $\mu_V = \frac1m \sum_{j=1}^m \delta_{v_j}$ be empirical measures on $\R$, $u_{(1)} \le \cdots \le u_{(n)}$ and $v_{(1)} \le \cdots \le v_{(m)}$. Let $0 = h_0 \leq h_1 \leq \cdots \leq h_{n+m} = 1$ where $h_1, \ldots h_{n+m}$  are the sorted values of $\{k/n\}_{k=0}^n \cup \{l/m\}_{l=0}^m$. Then for any $q \ge 1$,
\[
 W_q^q(U,V):=\int_{0}^1|F^{-1}_U(s) - F^{-1}_{V}(s)|^q \di s =
\sum_{i=1}^{n+m}
\bigl|u_{(\lceil n h_k \rceil)} - v_{(\lceil m h_k \rceil)}\bigr|^q \, (h_{k} - h_{k-1}).
\]
Furthermore,  if $\sigma_U, \sigma_V$ are ordering permutations of $U$ and $V$, an optimal transport plan $\gamma \in \Pi(\mu_U,\mu_V)$ is given for all $i,j$ by
$$\gamma_{\sigma_U(i), \sigma_V(j)}
:=
\sum_{k=1}^{n+m}
\mathds{1}_{\{\lceil n h_k\rceil=i\}}
\mathds{1}_{\{\lceil m h_k\rceil=j\}}
\,(h_k-h_{k-1}).$$
\end{lemma}
\begin{proof}[of \Cref{lem:unbalanced_wasserstein}] Let $k \in \{1,\ldots,n+m\}$. We look at the integrand on the interval 
$(h_{k-1}, h_k]$ whenever $h_{k-1} < h_k$. 
There exist $(i,j) \in \{0,\ldots,n-1\} \times \{0,\ldots,m-1\}$ satisfying one of the following cases:
\begin{itemize}
    \item $h_{k-1} = \frac{i}{n}$, $h_k = \frac{j+1}{m}$ where $\frac{i}{n} \leq \frac{j}{m} \leq \frac{i+1}{n} \leq \frac{j+1}{m}.$
    \item $h_{k-1} = \frac{j}{m}$, $h_k = \frac{i}{n}$ and $\frac{j}{m} \leq \frac{i}{n} \leq \frac{j+1}{m} \leq \frac{i+1}{n}.$
    \item $h_{k-1} = \frac{i}{n}$, $h_k = \frac{i+1}{n}$ and $\frac{j}{m} \leq \frac{i}{n} \leq \frac{i+1}{n} \leq \frac{j+1}{m}.$
    \item $h_{k-1} = \frac{j}{m}$, $h_k = \frac{j+1}{m}$ and $\frac{i}{n} \leq \frac{j}{m} \leq \frac{j+1}{m} \leq \frac{i+1}{n}.$
\end{itemize}
In each case, $F^{-1}_{U}$ and $F^{-1}_V$ are constant on $(h_{k-1}, h_k]$ for all $k \in \{1, \ldots, n+m\}$. Now, let $s \in (0,1]$. There exists $i \in \{ 0, \ldots,n-1\}$ such that $\frac{i}{n} < s \leq \frac{i+1}{n}$. On $\left(\frac{i}{n}, \frac{i+1}{n} \right]$ this implies $i+1 = \lceil n s\rceil$. Furthermore, $F_{U}^{-1}(s)$ is equal to $U_{(i+1)}$ on the same interval. Thus $F^{-1}_U(s) = U_{(\lceil ns \rceil)}$. Similarly, $F^{-1}_V(s) = V_{(\lceil ms \rceil)}$. This allows us to write
\begin{equation*}
    W_q(U,V)^q = \sum_{k = 1}^{n+m} \int_{h_{k-1}}^{h_{k}} \left| F^{-1}_U(s)- F^{-1}_V(s)\right|^q \di s = \sum_{k = 1}^{n+m} |U_{(\lceil n h_{k} \rceil)} - V_{(\lceil m h_{k} \rceil)}|^q (h_k - h_{k-1}).
\end{equation*}

For any $i \in \{1, \ldots,n\}$, $j \in \{1, \ldots,n\}$, we set $\pi_{ij}
:=
\sum_{k=1}^{n+m}
\mathds{1}_{\{\lceil n h_k\rceil=i\}}
\mathds{1}_{\{\lceil m h_k\rceil=j\}}
\,(h_k-h_{k-1}).$

Up to a permutation, let us show that this defines a transport plan. Remark that $\lceil n h_k\rceil = i$ if and only if $h_k \in \left(\frac{i-1}{n},\frac{i}{n} \right]$, which means $(h_{k-1}, h_k] \subset  \left(\frac{i-1}{n},\frac{i}{n} \right]$. Necessarily, $ \left(\frac{i-1}{n},\frac{i}{n} \right]$ is the disjoint union of intervals of the form $(h_{k-1}, h_k]$, $k = 1, \ldots,n+m$. Hence for any $i$ we have
\begin{equation*}
\sum_{j=1}^m\pi_{ij}=    \sum_{k=1}^{n+m} \mathds{1}_{\{\lceil nh_k\rceil= i \}}(h_{k} - h_{k-1}) = \frac{1}{n}.
\end{equation*}

The same argument holds when summing on indexes $i = 1, \ldots,n$, and $\sum_{i =1}^n \pi_{ij} = \frac{1}{m}$. We verify that $\pi$ is optimal, up to ordering permutations.
\begin{align*}
    \sum_{i=1}^n \sum_{j=1}^m \pi_{ij}|u_{(i)} - v_{(j)}|^q & = \sum_{k=1}^{n+m} \sum_{i=1}^n \sum_{j=1}^m (h_k - h_{k-1})|u_{(i)} - v_{(j)}|^q  \mathds{1}_{\{\lceil n h_k\rceil=i\}} 
\mathds{1}_{\{\lceil m h_k\rceil=j\}} \\ & = \sum_{k=1}^{n+m} (h_k - h_{k-1}) |u_{(\lceil n h_k\rceil)} - v_{(\lceil m h_k\rceil)}|^q = W_q^q(U,V)
\end{align*}
hence the result by taking $\gamma_{\sigma_U(i),\sigma_V(j)} = \pi_{ij}$.
\end{proof}

Now we can prove the calculus part of \Cref{cor:limit_subdiff_fair}.

\begin{lemma} \label{lem:correctness_fair} Under \Cref{ass:differentiability_fairness}, for any $n_0, n_1$ let $G^{\FR}_{n_0,n_1}$ be defined as \eqref{eq:gradient_oracle_fairness}. Then for all $\theta \in \R^p$, $G^{\FR}_{n_0,n_1}(\theta) \in \partial \FR_{n_0,n_1}(\theta)$.
\end{lemma}

\begin{proof} We set $C(\theta,x,y) =  \frac{1}{2}\big(
s(\theta,x)
-
s(\theta,y)
\big)^2$. For $j =0,1$, let $\sigma^j_\theta$  be a permutation of $\{1, \ldots,n_j\}$ such that $s(\theta,x^j_{(1)}) \leq \ldots \leq  s(\theta,x^j_{(n_j)})$. With these notations, we can write
\[
G_{n_0,n_1}^{\operatorname{FR}}(\theta)
=
\sum_{k=1}^{n_0+n_1} \nabla_{\theta} C \left(\theta,x^0_{\sigma^0_\theta(\lceil n_0 h_k\rceil)},x^1_{\sigma^1_\theta(\lceil n_1h_{k} \rceil)} \right) (h_{k}-h_{k-1}),
\]
Hence, the transport plan formulation
\begin{align}
   G_{n_0,n_1}^{\operatorname{FR}}(\theta)& =  \sum_{k=1}^{n_0 + n_1} \sum_{i=1}^{n_0} \sum_{j=1}^{n_1}  \nabla_{\theta} C \left(\theta,x^0_{\sigma^0_\theta(i)},x^1_{\sigma^1_\theta(j)} \right)(h_k - h_{k-1}) \mathds{1}_{\{\lceil n_0 h_k\rceil=i\}} 
\mathds{1}_{\{\lceil n_1 h_k\rceil=j\}} \nonumber \\
& = \sum_{i=1}^{n_0} \sum_{j=1}^{n_1} \gamma_{i,j} \nabla_{\theta} C \left(\theta,x^0_{i},x^1_{j} \right).
\label{eq:grad_fair_as_gamma_nabla_C}
\end{align}
where $\gamma$ is given by \Cref{lem:unbalanced_wasserstein_extended} and is optimal for the discrete transport problem 
\begin{equation*}
    \min_{\gamma \in \R^{n_0 \times n_1}} \frac{1}{2} \sum_{i=1}^{n_0} \sum_{j=1}^{n_1} \gamma_{ij}  \big(
s^0(\theta,x_i^0)
-
s^{1}(\theta,x^1_j))^2
\end{equation*}
By the envelope formula, \Cref{prop:envelope_ot}, \eqref{eq:grad_fair_as_gamma_nabla_C} reads $G_{n_0,n_1}^{\operatorname{FR}}(\theta) \in \partial \FR_{n_0, n_1}(\theta)$.
\end{proof}

Now we prove the limit result, which is the second part of \Cref{cor:limit_subdiff_fair}.

\begin{proof}[of  \Cref{cor:limit_subdiff_fair}] \label{proof:fair} We prove the graphical convergence of $\partial \FR_{n_0,n_1}$ to $\partial \FR$ on any compact set $\Theta \subset \R^p$. 
Our goal is to apply \Cref{th:limit_subdiff}. We verify each point of \Cref{assumption:regularity_cost}. We set $C(\theta,x,y) =  \frac{1}{2}\big(
s(\theta,x)
-
s(\theta,y)
\big)^2.$
Under \Cref{ass:differentiability_fairness}.2,
\begin{align*}
    C(\theta,x,y) \leq  s^2(\theta,x) + s^{2}(\theta,y) \leq \kappa^2(\theta) (\psi^2(x) + \psi^2(y)).
\end{align*}
Under \Cref{ass:differentiability_fairness}.2 we have $\int_{\X}\psi^{2(1+\eta)} \di \mu_j(x) < \infty$ for each group $j = 0,1$. Thus \Cref{assumption:regularity_cost}.2 holds with the replacements $\kappa \leftarrow \kappa^2$, $(\mu,\nu) \leftarrow (\mu^0, \mu^1)$, $\zeta \leftarrow 2 \zeta$. All conditions are satisfied to apply \Cref{th:limit_subdiff} hence $\D \left(\Graph_\Theta \partial \FR_{n_0, n_1}, \Graph_\Theta \partial \FR  \right) \xrightarrow[\min(n_0, n_1)]{} 0,$
and we have item 1 \Cref{cor:limit_subdiff_fair}  by refining $\Theta \leftarrow \Theta + e \Bar{\B}$ for some $e > 0$.

As to item 2, under \Cref{ass:differentiability_fairness}.3 we have by standard uniform law of large number $\D(\Graph_\Theta \nabla \mathcal{L}_n, \Graph_\Theta \nabla \mathcal{L}) \xrightarrow[n \to \infty]{} 0$. Hence by the sum rule of \Cref{lem:graphical_cv_sum},  $\nabla \mathcal{L}_n + \lambda \partial \FR_{n_0,n_1}$ graphically converge to $\nabla \mathcal{L} + \lambda \partial \FR$ on the compact set $\Theta$. Thus we may apply \Cref{lem:critical_set_outer} to the maps 
\begin{equation*}
    \nabla \mathcal{L}_n + \lambda \partial \FR_{n_0,n_1} + N_\Theta \qquad \text{and} \qquad \nabla \mathcal{L} + \lambda \partial \FR + N_\Theta 
\end{equation*}
to obtain $\D(\crit_\Theta (\mathcal{L}_{n_0 + n_1} + \lambda \FR_{n_0,n_1}) , \crit_\Theta (\mathcal{L} + \lambda \FR)) \xrightarrow[\min(n_0, n_1) \to \infty]{} 0$.
\end{proof}

\subsubsection{Sliced Wasserstein learning.} 

\begin{proof}[of \Cref{cor:limit_subdiff_sliced_wasserstein}] This is an application of \Cref{th:limit_subdiff_sliced} with the choice 

    $$C: (\theta,x,y, \phi) \mapsto \frac{1}{2}(\phi^\top f(\theta,x)-  \phi^\top y )^2$$

and $\nabla_\theta C (\theta,x,y,\phi) =\phi^\top (f(\theta,x) - y)  \Jac_\theta f(\theta,x)^\top \phi.$ Let us verify the boundedness conditions from \Cref{assumption:main_differentiability_parameterized_sliced_cost}. Using \(\|\phi\|=1\), we have $|C(\theta,x,y,\phi)|
\le
\frac12\big(\|f(\theta,x)\|+\|y\|\big)^2.$
Moreover,
\[
\|\nabla_\theta C(\theta,x,y,\phi)\|
\le
|\phi^\top(f(\theta,x)-y)|\,\|\Jac_\theta f(\theta,x)^\top\phi\|
\le
\big(\|f(\theta,x)\|+\|y\|\big)\|\Jac_\theta f(\theta,x)\|.
\]
By the growth assumption, $\|f(\theta,x)\|+\|\Jac_\theta f(\theta,x)\|
\le
\kappa(\theta)\psi_\X(x),$
and therefore, on any compact set \(\Theta\subset\mathbb R^p\), with
\(\kappa_\Theta:=\sup_{\theta\in \Theta}\kappa(\theta)<\infty\),
\[
|C(\theta,x,y)|+\|\nabla_\theta C(\theta,x,y)\|
\le
\frac12\big(\kappa_\Theta\psi_\X(x)+\|y\|\big)^2
+
\big(\kappa_\Theta\psi_\X(x)+\|y\|\big)\kappa_\Theta\psi_\X(x).
\]
Thus, for some constant \(A_\Theta>0\), $|C|+\|\nabla_\theta C\|
\le
A_\Theta\big(1+\psi_\X(x)^2+\|y\|^2\big).$
Since $\int_\X \psi_\X^{2(1+\eta)}\,d\mu<\infty$ and $
\int_{\mathbb R^d}\|y\|^{2(1+\eta)}\,d\nu<\infty,$
This verifies the boundedness condition, and the result follows.
\end{proof}
\section{Failure cases for nonsmooth unit costs}
\label{sec:failure}

In this section, we exhibit several situations involving nonsmooth unit costs and nonsmooth models in which our results fail. In \Cref{subsec:failure_envelope}, we give an example where the envelope formula (\Cref{prop:envelope_ot}) does not apply, and hence does not correctly compute subgradients. In \Cref{subsec:non_convergence}, we then present examples where subdifferential convergence for empirical objectives fails; in fact, the empirical objectives may produce spurious local minima that persist in the large-sample limit.

\subsection{Failure of the envelope formula}
\label{subsec:failure_envelope}
The envelope formula in \Cref{prop:envelope_ot} does not apply to general nonsmooth costs \(C\) without additional structural assumptions, such as local weak concavity \cite{clarke1990optimization}. In particular, when \(C\) is defined as in \eqref{eq:c_f_g} and the unit cost \(c\) is a distance, as in Wasserstein-1, the assumptions underlying our approach may fail. The following example illustrates this limitation through a failure of the envelope formula. Let $\mu=\nu=\operatorname{Unif}(\mathbb S^1),$ where \(\mathbb S^1\subset\mathbb R^2\) is the one-dimensional circle, and consider the smooth family of rotations.
\[
f(\theta,x) = R(\theta) x,
\qquad
R(\theta)=
\begin{pmatrix}
\cos\theta & -\sin\theta\\
\sin\theta & \cos\theta
\end{pmatrix}.
\]
We equip \(\mathbb S^1\) with the intrinsic \emph{arclength distance}
\[
d_{\mathbb S^1}(x,y)
=
\min_{k\in\mathbb Z}
|\alpha(x)-\alpha(y)+2\pi k|,
\]
where \(\alpha(x)\in [0, 2\pi)\) denotes the angular coordinate of \(x\). Since rotations preserve \(\mu\), $(f_\theta)_\#\mu=\nu$ for all $\theta.$ Hence the population transport value is zero everywhere
\[
\T(\theta)
:=
\min_{\gamma\in\Pi(\mu,\nu)}
\int_{\mathbb S^1\times\mathbb S^1}
d_{\mathbb S^1}(R(\theta) x,y)\,\di\gamma(x,y) \equiv0.
\]
Indeed, the plan $\gamma_\theta=(\operatorname{Id},R(\theta) \cdot)_\#\mu$
is admissible and has zero cost. However, the naive nonsmooth envelope expression gives a larger set. On the support of
\(\gamma_\theta\), one has \(y=f(\theta,x)\), hence
\[
d_{\mathbb S^1}(f(\theta, x),y)
=
\min_{k\in\mathbb Z}|\alpha(x)+\theta-\alpha(y)+2\pi k|.
\]
At points where \(y=R_\theta x\), the argument of the absolute value is zero. Since the subdifferential of \(r\mapsto |r|\) at \(0\) is \([-1,1]\), we obtain
$\partial_\theta d_{\mathbb S^1}(f(\theta,x),y)
=
[-1,1]$ $\gamma_\theta$-almost everywhere. Consequently, at any $\theta$,
\[
[-1,1]=
\left\{
\int \zeta(x,y)\,\di\gamma_\theta(x,y)
:\ 
\zeta(x,y)\in
\partial_\theta d_{\mathbb S^1}(R_\theta x,y)
\right\},
\]
whereas the true derivative is the singleton \(\{0\}\). Thus, for this nonsmooth distance cost, one cannot rely on the pointwise subgradient envelope formula used above to compute subgradients. 

\subsection{Nonconvergence of empirical derivatives}
\label{subsec:non_convergence}

We show that empirical subgradient convergence, as stated in \Cref{th:limit_subdiff}, may fail for nonsmooth costs and models. We first present a simple example based on the ReLU function, and then build on it to obtain a spurious critical point in the large-sample limit.

\paragraph{A ReLU example.}

Define $h(\theta,x)=x + \operatorname{relu}(\theta)$, where \(\operatorname{relu}(\theta)=\max\{\theta,0\}\). Let
\(\mu_\theta=(h_\theta)_\#\mu\), with  $h_\theta = h(\theta,\cdot)$ and \(\mu=\operatorname{Unif}(0,1)\), and take
\(\nu=\operatorname{Unif}(0,1)\) as a target distribution. Then
\[
\T(\theta):=W_1(\mu_\theta,\nu)=\operatorname{relu}(\theta).
\]
Indeed, for \(\theta\le 0\), \(\mu_\theta=\nu\), while for \(\theta>0\),
\(\mu_\theta=\operatorname{Unif}(\theta,1+\theta)\), so the optimal transport is the
translation by \(\theta\). In particular  \(\T\) is nonsmooth at \(0\) and
\[
\partial \T(0)=[0,1].
\]
For i.i.d and mutually independent $x_1, \ldots x_n \sim \mu$, $y_1,\ldots,y_n \sim \nu$, The one-dimensional sorting formula gives the empirical objective
\[
\T_n(\theta)=\frac1n\sum_{i=1}^n |x_{(i)}+\theta-y_{(i)}|, \qquad \text{for } \theta  >0
\]
and $\T_n(\theta)=\frac1n\sum_{i=1}^n |x_{(i)}-y_{(i)}|$ for $\theta \leq 0$.

While population subderivatives belong to $[0,1]$, limits of empirical subderivatives may approach all values in $[-1,1]$, as stated below.
\begin{proposition} As $n \to \infty$, $\partial \T_n(0)$ converges in law to a random segment $\{tU \ : t \in [0,1]\}$ with $U \sim \operatorname{Unif}(-1,1)$.
\end{proposition}

\begin{proof} Almost surely, for each $i$, $x_{(i)}  \neq  y_{(i)}$. On this event the right derivative at \(0\) is $\T'_{n,+}(0)=\frac1n\sum_{i=1}^n \operatorname{sgn}(x_{(i)}-y_{(i)}).$
By Donsker's theorem (quantiles of uniform variables are also sorted cumulative distribution values), we have the convergence in law (uniformly in $s$)
\[
\sqrt n\big(x_{(\lceil ns\rceil)}-y_{(\lceil ns\rceil)}\big)
\xrightarrow[n \to \infty]{} B_1(s)-B_2(s),
\]
where \(B_1,B_2\) are independent Brownian bridges. Since multiplication by
\(\sqrt n>0\) does not change signs, and since the zero set of \(B_1-B_2\) has
Lebesgue measure zero almost surely (hence avoiding discontinuity point of $\operatorname{sgn}$), the continuous mapping theorem gives the convergence in law
\[
\frac1n\sum_{i=1}^n \operatorname{sgn}(\sqrt{n}(x_{(i)}-y_{(i)})) = \T'_{n,+}(0) 
\xrightarrow[n \to \infty]{}
\int_0^1 \operatorname{sgn}(B_1(s)-B_2(s))\,ds.
\]
As \(B_1-B_2\) has the same sign process as a Brownian bridge, Lévy's uniform law for
the positive occupation time of the Brownian bridge \cite{levy1940certains} gives $\int_0^1 \operatorname{sgn}(B_1(s)-B_2(s))\,ds
\sim \operatorname{Unif}(-1,1).$ We easily see the left derivative of $\T_n$ is always $0$ hence the result.
\end{proof}
\paragraph{Spurious criticality.}
The previous example shows that nonsmooth parameterizations may generate empirical
first-order behavior that is not captured by the population Clarke subdifferential.
We now strengthen this observation by constructing a point that is not critical for
the population objective, but which belongs to the limiting support of empirical
subdifferentials with significant probability.

Fix \(w\in(1/2,1)\), for instance \(w=3/4\), and choose \(M>4\). Let
\[
\mu=w\,\operatorname{Unif}(0,1)+(1-w)\delta_M,
\qquad
\nu=w\,\operatorname{Unif}(0,1)+(1-w)\delta_{M-2}.
\]
We use a continuous interpolation between a ReLU shift on the continuous component and
a linear shift near the atom. Namely, let \(\chi:\mathbb R\to[0,1]\) be continuous,
with \(\chi=1\) on a neighborhood of \([0,1]\) and \(\chi=0\) on a neighborhood of \(M\),
and define
\[
h(\theta,x)
=
x+\chi(x)\operatorname{relu}(\theta) +(1-\chi(x))\theta,
\qquad
\operatorname{relu}(\theta)=\max\{\theta,0\}.
\]
Thus \(h(\theta,x)=x+\operatorname{relu}(\theta)\) on the continuous component, while
\(h(\theta,M)=M+\theta\). Let
\[
\T(\theta)=W_1((h_\theta)_\#\mu,\nu) := \int_{0}^1 |F^{-1}_{(h_\theta)_{\#}\mu}(s)- F^{-1}_{\nu}(s)| \di s
\]
Since $M$ is high enough, and when $|\theta|$ is small enough, the monotone transport map matches the continuous block with
the continuous block and the atom with the atom. We easily verify that for $|\theta|$ small enough,
\[
\T(\theta)=w\operatorname{relu}(\theta) +(1-w)|2+\theta|,
\]
hence in particular $\partial \T(0)=[1-w,1]$ and \(0\notin\partial \T(0)\). But we show that $0$ can be a critical point of the empirical objectives, even for large sample sizes.

\begin{proposition}[Spurious empirical criticality] \label{prop:spurious} Let $\mu_n$, $\nu_n$ be i.i.d empirical distributions associated to $\mu$ and $\nu$. and  let $T_n(\theta)=W_1((h_\theta)_\#\mu_n,\nu_n).$ Then
$\partial T_n(0)$ converges in law to $
\overline{\operatorname{conv}}\{1-w,\;1-w+wU\},$ $
U\sim\operatorname{Unif}(-1,1).$
In particular, $\mathbb P\bigl(0\in\partial \T_n(0)\bigr)
\longrightarrow
1 - \frac{1}{2w}>0.$
\end{proposition}

\begin{proof}
Let $x_1,\ldots,x_n \sim \mu$ and $y_1, \ldots y_n \sim \nu$ be the i.i.d samples and define
\[
N_X=\#\{i:x_i\in[0,1]\},
\qquad
N_Y=\#\{i:y_i\in[0,1]\}.
\]
Then \(N_X,N_Y\) are independent binomial variables of parameters $(n,w)$. Hence
\(N_X/n\to w\) and \(N_Y/n\to w\) almost surely by law of large numbers. Thus $\frac{|N_X-N_Y|}{n} \to 0$ almost surely. Set \(m_n=\min(N_X,N_Y)\). Since the continuous block and the atom are separated,
monotone one-dimensional transport matches the common continuous mass with the
continuous mass and the atom mass with the atom mass. Only the excess
\(|N_X-N_Y|\) points can be matched across blocks, and this has total mass
\(|N_X-N_Y|/n\). Since \(|h'_+(0,\cdot)|\le1\) on the support of
\(\mu\), this cross-block part contributes only  to the one-sided
derivatives. For the left derivative, only the Dirac components contribute and it writes
\[
\T'_{n,-}(0)=\frac{n-N_X}{n}+\varepsilon_n \to 1 - w \qquad \text{almost surely,}
\]
where by the previous argument $\varepsilon_n \to 0$ almost surely. Let \(x^0_{(1)}\le\cdots\le x^0_{(N_X)}\) and
\(y^0_{(1)}\le\cdots\le y^0_{(N_Y)}\) be the ordered  samples falling in the continuous components. Then the right derivative writes
\[
\T'_{n,+}(0)
=
\frac{n-N_X}{n}
+
\frac1n\sum_{i=1}^{m_n}
\operatorname{sgn}(x^0_{(i)}-y^0_{(i)})
+ \zeta_n
\]
where $\zeta_n \to 0$ almost surely. The first term converges to $1 - w$. It remains to identify the averaged sign term. Conditionally on \(N_X,N_Y\), the continuous-block samples are independent iid uniform samples, hence as in the previous example,
\[
D_m:=\frac1m\sum_{i=1}^m \operatorname{sgn}(x_{(i)}^{0}-y_{(i)}^{0})
\xrightarrow[m \to \infty]{} U,
\qquad
U\sim\operatorname{Unif}(-1,1).
\]
Since \(m_n/n\to w\) almost surely, we have the weak convergence
\[
\frac1n\sum_{i=1}^{m_n}
\operatorname{sgn}(x^0_{(i)}-y^0_{(i)})
=
\frac{m_n}{n}D_{m_n}
\rightarrow wU.
\]
This finally gives
\[
\T'_{n,-}(0)\to 1-w,
\qquad
\T'_{n,+}(0)\rightarrow 1-w+wU.
\]
where $U\sim\operatorname{Unif}(-1,1)$. Finally, \(0\in\overline{\operatorname{conv}}\{1-w,1-w+wU\}\) if and only if
\(1-w+wU\le0\). Thus, when $w \in (1/2, 1)$, $-1 < 1 - 1/w < 0$ and  we have
\[
\mathbb P(1-w+wU\le0)
=
\mathbb P\left(U\le 1 - \frac{1}{w}\right)
=1 - \frac{1}{2w}.
\]
This proves the claim.
\end{proof}

\begin{wrapfigure}{r}{0.42\textwidth}
    \centering
    \vspace{-0.5em}
    \includegraphics[width=0.40\textwidth]{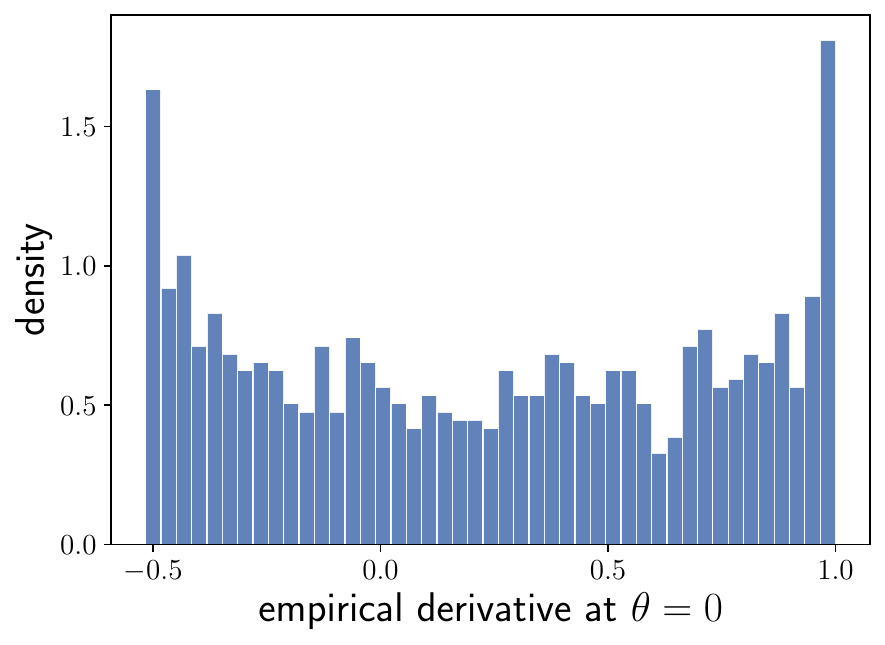}
    \caption{\small{The population objective is increasing, with subgradients in $[\frac 14, 1]$ while the range of empirical derivatives contains zero.
    }}
    \label{fig:nonsmooth-empirical-derivatives}
    \vspace{-1em}
\end{wrapfigure}
\paragraph{Experiment.} For the numerical illustrations, we take $w = 3/4$ and $M = 6$.

In \Cref{fig:spurious-local-minimum}, we plot the population objective together
with several realizations of the empirical objective \(\T_n\), for \(n=10\,000\) (grey curves).
Although the population objective is increasing near \(\theta=0\), some empirical
realizations decrease to the right of \(0\) and develop a local minimum.
This illustrates how the limiting empirical first-order behavior can create spurious
criticality even when the population Clarke subdifferential is bounded away from zero.

In \Cref{fig:nonsmooth-empirical-derivatives}, we repeat the experiment \(1000\) times
and display the empirical distribution of the right derivative at \(\theta=0\). For each
realization, \(\T_n\) is computed by the one-dimensional sorting formula, and the
derivative is evaluated by automatic differentiation. The observed values uniformly range
approximately from \(-0.5\) to \(1\), in agreement
with the limiting random segment described in \Cref{prop:spurious}.

\section{Conclusion}

This work studied parameterized empirical transport costs from the joint perspective of statistical approximation and nonsmooth optimization. We established graphical convergence of subdifferentials for a broad class of nonsmooth objectives generated by differentiable transport costs. These results provide a framework for understanding when first-order information computed from sampled transport losses is consistent with the corresponding population objective.

Our examples show that the smoothness of the parameterization plays an important role in obtaining favorable statistical and optimization guarantees. In contrast, general nonsmooth costs may exhibit unstable empirical derivative behavior in the large-sample limit. This may indicate that nonsmooth transport losses require a more refined first-order description than direct convergence of classical subdifferentials. Objective-decoupled gradient models~\cite{norkin1980,bolte2019conservative,schechtman2024gradient} provide a natural language for such limiting objects, although the distributional nature of transport objectives makes their direct application challenging. A complementary direction is to understand whether mini-batching and smoothing-by-sampling mechanisms that are used in practice can restore stable first-order behavior~\cite{sampling_fatras2021minibatch,sliced_tanguy2025properties}.
\bibliographystyle{siam}
\bibliography{references}

\begin{thebibliography}{10}

\bibitem{infinite}
{\sc C.~Aliprantis and K.~Border}, {\em Infinite Dimensional Analysis},
  Springer Berlin, Heidelberg, 2006.

\bibitem{ambrosio2005gradient}
{\sc L.~Ambrosio, N.~Gigli, and G.~Savar{\'e}}, {\em Gradient flows: in metric
  spaces and in the space of probability measures}, Springer, 2005.

\bibitem{arjovsky2017wasserstein}
{\sc M.~Arjovsky, S.~Chintala, and L.~Bottou}, {\em Wasserstein generative
  adversarial networks}, in International conference on machine learning, Pmlr,
  2017, pp.~214--223.

\bibitem{graphical_artstein1975strong}
{\sc Z.~Artstein and R.~A. Vitale}, {\em A strong law of large numbers for
  random compact sets}, The Annals of Probability,  (1975), pp.~879--882.

\bibitem{graphical_attouch2006convergence}
{\sc H.~Attouch}, {\em Convergence de fonctionnelles convexes}, in Journ\'{e}es
  d'Analyse Non Lin\'{e}aire: Proceedings, Besan\c{c}on, France, June 1977,
  Springer, 2006, pp.~1--40.

\bibitem{aubin1987graphical}
{\sc J.-P. Aubin}, {\em Graphical convergence of set-valued maps},  (1987).

\bibitem{benaim2012}
{\sc M.~Bena{\"i}m, J.~Hofbauer, and S.~Sorin}, {\em Perturbations of
  set-valued dynamical systems, with applications to game theory}, Dynamic
  Games and Applications, 2 (2012), pp.~195--205.

\bibitem{billingsley2013convergence}
{\sc P.~Billingsley}, {\em Convergence of probability measures}, John Wiley \&
  Sons, 2013.

\bibitem{bolte2019conservative}
{\sc J.~Bolte and E.~Pauwels}, {\em Conservative set valued fields, automatic
  differentiation, stochastic gradient methods and deep learning}, Mathematical
  Programming, 188 (2021), pp.~19--51.

\bibitem{risk_bonalli2025characterization}
{\sc R.~Bonalli, B.~Bonnet-Weill, and L.~Pfeiffer}, {\em A characterization of
  law-invariant and coherent risk measures through optimal transport}, arXiv
  preprint arXiv:2512.19157,  (2025).

\bibitem{carlier2017convergence}
{\sc G.~Carlier, V.~Duval, G.~Peyr{\'e}, and B.~Schmitzer}, {\em Convergence of
  entropic schemes for optimal transport and gradient flows}, SIAM Journal on
  Mathematical Analysis, 49 (2017), pp.~1385--1418.

\bibitem{sliced_chapel2026differentiable}
{\sc L.~Chapel, R.~Tavenard, and S.~Vaiter}, {\em Differentiable generalized
  sliced wasserstein plans}, Advances in Neural Information Processing Systems,
  38 (2026), pp.~162905--162929.

\bibitem{clarke1990optimization}
{\sc F.~Clarke}, {\em Optimization and Nonsmooth Analysis}, Classics in Applied
  Mathematics, Society for Industrial and Applied Mathematics, 1990.

\bibitem{clarke1975generalized}
{\sc F.~H. Clarke}, {\em Generalized gradients and applications}, Transactions
  of the American Mathematical Society, 205 (1975), pp.~247--262.

\bibitem{sinkhorn_cuturi2014fast}
{\sc M.~Cuturi and A.~Doucet}, {\em Fast computation of wasserstein
  barycenters}, in International conference on machine learning, PMLR, 2014,
  pp.~685--693.

\bibitem{cuturi2022optimal}
{\sc M.~Cuturi, L.~Meng-Papaxanthos, Y.~Tian, C.~Bunne, G.~Davis, and
  O.~Teboul}, {\em Optimal transport tools (ott): A jax toolbox for all things
  wasserstein}, arXiv preprint arXiv:2201.12324,  (2022).

\bibitem{sinkhorn_NEURIPS2019_d8c24ca8}
{\sc M.~Cuturi, O.~Teboul, and J.-P. Vert}, {\em Differentiable ranking and
  sorting using optimal transport}, in Advances in Neural Information
  Processing Systems, H.~Wallach, H.~Larochelle, A.~Beygelzimer,
  F.~d\textquotesingle Alch\'{e}-Buc, E.~Fox, and R.~Garnett, eds., vol.~32,
  Curran Associates, Inc., 2019.

\bibitem{villani}
{\sc V.~Cédric}, {\em Optimal transport : old and new / Cédric Villani},
  Grundlehren der mathematischen Wissenschaften, Springer, Berlin, 2009.

\bibitem{danskin2012theory}
{\sc J.~M. Danskin}, {\em The theory of max-min and its application to weapons
  allocation problems}, Springer Science \& Business Media, 2012.

\bibitem{Davis2020}
{\sc D.~Davis, D.~Drusvyatskiy, S.~Kakade, and J.~D. Lee}, {\em Stochastic
  subgradient method converges on tame functions}, Foundations of Computational
  Mathematics, 20 (2020), pp.~119--154.

\bibitem{dellacherie2011probabilities}
{\sc C.~Dellacherie and P.-A. Meyer}, {\em Probabilities and potential, c:
  potential theory for discrete and continuous semigroups}, vol.~151, Elsevier,
  2011.

\bibitem{image_delon2004midway}
{\sc J.~Delon}, {\em Midway image equalization}, Journal of Mathematical
  Imaging and Vision, 21 (2004), pp.~119--134.

\bibitem{sliced_deshpande2018generative}
{\sc I.~Deshpande, Z.~Zhang, and A.~G. Schwing}, {\em Generative modeling using
  the sliced wasserstein distance}, in Proceedings of the IEEE conference on
  computer vision and pattern recognition, 2018, pp.~3483--3491.

\bibitem{dudley2018real}
{\sc R.~M. Dudley}, {\em Real analysis and probability}, Chapman and Hall/CRC,
  2018.

\bibitem{dumont2025existence}
{\sc T.~Dumont, T.~Lacombe, and F.-X. Vialard}, {\em On the existence of monge
  maps for the gromov--wasserstein problem}, Foundations of Computational
  Mathematics, 25 (2025), pp.~463--510.

\bibitem{durrett2010probability}
{\sc R.~Durrett}, {\em Probability: Theory and Examples}, Cambridge Series in
  Statistical and Probabilistic Mathematics, Cambridge University Press, 2010.

\bibitem{fairness_dwork2012fairness}
{\sc C.~Dwork, M.~Hardt, T.~Pitassi, O.~Reingold, and R.~Zemel}, {\em Fairness
  through awareness}, in Proceedings of the 3rd innovations in theoretical
  computer science conference, 2012, pp.~214--226.

\bibitem{sampling_fatras2021minibatch}
{\sc K.~Fatras, Y.~Zine, S.~Majewski, R.~Flamary, R.~Gribonval, and N.~Courty},
  {\em Minibatch optimal transport distances; analysis and applications}, arXiv
  preprint arXiv:2101.01792,  (2021).

\bibitem{fairness_feldman2015certifying}
{\sc M.~Feldman, S.~A. Friedler, J.~Moeller, C.~Scheidegger, and
  S.~Venkatasubramanian}, {\em Certifying and removing disparate impact}, in
  proceedings of the 21th ACM SIGKDD international conference on knowledge
  discovery and data mining, 2015, pp.~259--268.

\bibitem{flamary2021pot}
{\sc R.~Flamary, N.~Courty, A.~Gramfort, M.~Z. Alaya, A.~Boisbunon, S.~Chambon,
  L.~Chapel, A.~Corenflos, K.~Fatras, N.~Fournier, L.~Gautheron, N.~T. Gayraud,
  H.~Janati, A.~Rakotomamonjy, I.~Redko, A.~Rolet, A.~Schutz, V.~Seguy, D.~J.
  Sutherland, R.~Tavenard, A.~Tong, and T.~Vayer}, {\em Pot: Python optimal
  transport}, Journal of Machine Learning Research, 22 (2021), pp.~1--8.

\bibitem{follmer2011stochastic}
{\sc H.~F{\"o}llmer and A.~Schied}, {\em Stochastic finance: an introduction in
  discrete time}, Walter de Gruyter, 2011.

\bibitem{fournier2015rate}
{\sc N.~Fournier and A.~Guillin}, {\em On the rate of convergence in
  wasserstein distance of the empirical measure}, Probability theory and
  related fields, 162 (2015), pp.~707--738.

\bibitem{gao2023wassersteindistance}
{\sc R.~Gao and A.~Kleywegt}, {\em Distributionally robust stochastic
  optimization with wasserstein distance}, Math. Oper. Res., 48 (2023),
  pp.~603--655.

\bibitem{ghossoub2021continuity}
{\sc M.~Ghossoub and D.~Saunders}, {\em On the continuity of the feasible set
  mapping in optimal transport}, Economic Theory Bulletin, 9 (2021),
  pp.~113--117.

\bibitem{generative_gulrajani2017improved}
{\sc I.~Gulrajani, F.~Ahmed, M.~Arjovsky, V.~Dumoulin, and A.~C. Courville},
  {\em Improved training of wasserstein gans}, Advances in neural information
  processing systems, 30 (2017).

\bibitem{nonsmooth_houdard2023gradient}
{\sc A.~Houdard, A.~Leclaire, N.~Papadakis, and J.~Rabin}, {\em On the gradient
  formula for learning generative models with regularized optimal transport
  costs}, Transactions on Machine Learning Research,  (2023).

\bibitem{kuhn2019wasserstein}
{\sc D.~Kuhn, P.~M. Esfahani, V.~A. Nguyen, and S.~Shafieezadeh-Abadeh}, {\em
  Wasserstein distributionally robust optimization: Theory and applications in
  machine learning}, in Operations research \& management science in the age of
  analytics, Informs, 2019, pp.~130--166.

\bibitem{spectral_risk_laguel2022superquantile}
{\sc Y.~Laguel, J.~Malick, and Z.~Harchaoui}, {\em Superquantile-based
  learning: a direct approach using gradient-based optimization}, Journal of
  Signal Processing Systems, 94 (2022), pp.~161--177.

\bibitem{letrouit2024gluing}
{\sc C.~Letrouit and Q.~M{\'e}rigot}, {\em Gluing methods for quantitative
  stability of optimal transport maps}, arXiv preprint arXiv:2411.04908,
  (2024).

\bibitem{graphical_levy1995partial}
{\sc A.~B. Levy, R.~Poliquin, and L.~Thibault}, {\em Partial extensions of
  attouch's theorem with applications to proto-derivatives of subgradient
  mappings}, Transactions of the American Mathematical Society, 347 (1995),
  pp.~1269--1294.

\bibitem{levy1940certains}
{\sc P.~L{\'e}vy}, {\em Sur certains processus stochastiques homog{\`e}nes},
  Compositio mathematica, 7 (1940), pp.~283--339.

\bibitem{sliced_lobashev2026color}
{\sc A.~Lobashev, M.~Larchenko, and D.~Guskov}, {\em Color conditional
  generation with sliced wasserstein guidance}, Advances in Neural Information
  Processing Systems, 38 (2026), pp.~164572--164601.

\bibitem{risk_mehta2023stochastic}
{\sc R.~Mehta, V.~Roulet, K.~Pillutla, L.~Liu, and Z.~Harchaoui}, {\em
  Stochastic optimization for spectral risk measures}, in International
  Conference on Artificial Intelligence and Statistics, PMLR, 2023,
  pp.~10112--10159.

\bibitem{stability_merigot2020quantitative}
{\sc Q.~M{\'e}rigot, A.~Delalande, and F.~Chazal}, {\em Quantitative stability
  of optimal transport maps and linearization of the 2-wasserstein space}, in
  International Conference on Artificial Intelligence and Statistics, PMLR,
  2020, pp.~3186--3196.

\bibitem{sliced_nadjahi2021sliced}
{\sc K.~Nadjahi}, {\em Sliced-Wasserstein distance for large-scale machine
  learning: theory, methodology and extensions}, PhD thesis, Institut
  polytechnique de Paris, 2021.

\bibitem{sampling_nadjahi2020statistical}
{\sc K.~Nadjahi, A.~Durmus, L.~Chizat, S.~Kolouri, S.~Shahrampour, and
  U.~Simsekli}, {\em Statistical and topological properties of sliced
  probability divergences}, Advances in Neural Information Processing Systems,
  33 (2020), pp.~20802--20812.

\bibitem{nguyen2024sliced}
{\sc K.~Nguyen, S.~Zhang, T.~Le, and N.~Ho}, {\em Sliced wasserstein with
  random-path projecting directions}, in Proceedings of the 41st International
  Conference on Machine Learning, ICML'24, JMLR.org, 2024.

\bibitem{norkin1980}
{\sc V.~Norkin}, {\em Generalized-differentiable functions}, Cybernetics and
  Systems Analysis, 16 (1980), pp.~10--12.

\bibitem{NorWet13}
{\sc V.~I. Norkin and R.~J.-B. Wets}, {\em On a strong graphical law of large
  numbers for random semicontinuous mappings}, Vestnik S.-Petersburg
  University. Series 10. Applied Mathematics, Computer Science, Control
  Processes,  (2013), pp.~102--111.

\bibitem{sinkhorn_pauwels2023derivatives}
{\sc E.~Pauwels and S.~Vaiter}, {\em The derivatives of sinkhorn--knopp
  converge}, SIAM Journal on Optimization, 33 (2023), pp.~1494--1517.

\bibitem{peyre2019comput}
{\sc G.~Peyr\'{e} and M.~Cuturi}, {\em Computational optimal transport: With
  applications to data science}, Found. Trends Mach. Learn., 11 (2019),
  p.~355–607.

\bibitem{pillutla2024federated}
{\sc K.~Pillutla, Y.~Laguel, J.~Malick, and Z.~Harchaoui}, {\em Federated
  learning with superquantile aggregation for heterogeneous data}, Machine
  Learning, 113 (2024), pp.~2955--3022.

\bibitem{sliced_rabin2011wasserstein}
{\sc J.~Rabin, G.~Peyr{\'e}, J.~Delon, and M.~Bernot}, {\em Wasserstein
  barycenter and its application to texture mixing}, in International
  conference on scale space and variational methods in computer vision,
  Springer, 2011, pp.~435--446.

\bibitem{fairness_risser2022tackling}
{\sc L.~Risser, A.~G. Sanz, Q.~Vincenot, and J.-M. Loubes}, {\em Tackling
  algorithmic bias in neural-network classifiers using wasserstein-2
  regularization}, Journal of Mathematical Imaging and Vision, 64 (2022),
  pp.~672--689.

\bibitem{Rockafellar_1998}
{\sc R.~T. Rockafellar and R.~J.~B. Wets}, {\em Variational Analysis}, Springer
  Berlin Heidelberg, 1998.

\bibitem{sliced_rodriguez2025learning}
{\sc D.~Rodr{\'\i}guez-V{\'\i}tores, C.~Lalanne, and J.-M. Loubes}, {\em
  Learning with differentially private (sliced) wasserstein gradients}, arXiv
  preprint arXiv:2502.01701,  (2025).

\bibitem{fairness_rychener2022metrizing}
{\sc Y.~Rychener, B.~Taskesen, and D.~Kuhn}, {\em Metrizing fairness}, arXiv
  preprint arXiv:2205.15049,  (2022).

\bibitem{graphical_salim2023strong}
{\sc A.~Salim}, {\em A strong law of large numbers for random monotone
  operators}, Set-Valued and Variational Analysis, 31 (2023), p.~38.

\bibitem{Santambrogio2015}
{\sc F.~Santambrogio}, {\em Optimal Transport for Applied Mathematicians},
  Progress in Nonlinear Differential Equations and Their Applications,
  Birkh{\"a}user Cham, 1~ed., 2015.

\bibitem{schechtman2024gradient}
{\sc S.~Schechtman}, {\em The gradient’s limit of a definable family of
  functions admits a variational stratification}, SIAM Journal on Optimization,
   (2026).

\bibitem{sebbouh2022randomized}
{\sc O.~Sebbouh, M.~Cuturi, and G.~Peyr{\'e}}, {\em Randomized stochastic
  gradient descent ascent}, in International Conference on Artificial
  Intelligence and Statistics, PMLR, 2022, pp.~2941--2969.

\bibitem{shapiro}
{\sc A.~Shapiro and H.~Xu}, {\em Uniform laws of large numbers for set-valued
  mappings and subdifferentials of random functions}, Journal of Mathematical
  Analysis and Applications, 325 (2007), pp.~1390--1399.

\bibitem{lift_tanguy2025sliced}
{\sc E.~Tanguy, L.~Chapel, and J.~Delon}, {\em Sliced optimal transport plans},
  arXiv preprint arXiv:2508.01243,  (2025).

\bibitem{sliced_tanguy2025properties}
{\sc E.~Tanguy, R.~Flamary, and J.~Delon}, {\em Properties of discrete sliced
  wasserstein losses}, Mathematics of Computation, 94 (2025), pp.~1411--1465.

\bibitem{sliced_vauthier2025towards}
{\sc C.~Vauthier, A.~Korba, and Q.~M{\'e}rigot}, {\em Towards understanding
  gradient dynamics of the sliced-wasserstein distance via critical point
  analysis}, arXiv preprint arXiv:2502.06525,  (2025).

\bibitem{wang2023sinkhorn}
{\sc J.~Wang, R.~Gao, and Y.~Xie}, {\em Sinkhorn distributionally robust
  optimization}, 2023.

\bibitem{spectral_risk_xiao2023unified}
{\sc R.~Xiao, Y.~Ge, R.~Jiang, and Y.~Yan}, {\em A unified framework for
  rank-based loss minimization}, Advances in Neural Information Processing
  Systems, 36 (2023), pp.~51302--51326.

\bibitem{zolezzi1994convergence}
{\sc T.~Zolezzi}, {\em Convergence of generalized gradients}, Set-Valued
  Analysis, 2 (1994), pp.~381--393.

\end{thebibliography}
\end{document}